\documentclass[11pt,reqno]{amsart}
\usepackage{amsmath}
\usepackage{mathrsfs}
\usepackage{amsfonts}
\usepackage{amsthm}
\usepackage{amssymb}
\usepackage{amscd}
\usepackage{graphicx}
\usepackage{color}
\input xy
\xyoption{all} \theoremstyle{plain}

\newtheorem{theo}{Theorem}[section]
\newtheorem{lemma}[theo]{Lemma}
\newtheorem{definition}[theo]{\bf Definition}
\newtheorem{proposition}[theo]{\bf Proposition}
\newtheorem{cor}[theo]{\bf Corollary}
\newtheorem{remark}[theo]{Remark}
\newtheorem{parrafo}[theo]{}
\newtheorem*{parraf}{}%
\newtheorem{example}[theo]{Example}
\newtheorem{paragr}{}[subsubsection]
\newcommand{\grad}{\textnormal{grad}}
\newcommand{\Gr}{\operatorname{Gr}}
\newcommand{\rank}{\textnormal{rank}}
\newcommand{\codim}{\textnormal{codim}}
\newcommand{\Jac}{\textnormal{Jac}}
\newcommand{\Aut}{\textnormal{Aut}}
\newcommand{\im}{\textnormal{Im}}

\begin{document}
\title[Moduli spaces of coherent systems]
{Hodge polynomials of some moduli spaces of coherent systems}
\thanks{This work has been partially supported by two EC Training Fellowships. The first one was within
the ``Liverpool Mathematics International Training Site'' (LIMITS)
supported as a Marie Curie Training Site of the European Community
Programme ``Improving Human Research Potential and the
Socio-Economic Knowledge Base'' Contract No.
HPMT-CT-2001-00277. The second one was within the Research Training
Network LIEGRITS: Flags, Quivers and Invariant Theory in Lie
Representation Theory, which is a Marie Curie Research Training
Network funded by the European community as project MRTN-CT
2003-505078}  \dedicatory{Dedicated to Peter Newstead in testimony of friendship and gratitude}

\author{Cristian Gonz\'alez--Mart\'inez}
\address{{\it Formerly at:} Department of Mathematics\\Tufts University\\Bromfield-Pearson Building\\503 Boston
Avenue\\
Medford, MA 02155\\USA.}
\address{{\it Currently at:} DG-Payments and Market Infrastructure\\European Central Bank\\Neue Mainzer Stra$\beta$e\\ 60311 Frankfurt am Main\\Germany.}
\email{c.gonzalez-martinez@hotmail.com}

\maketitle

\begin{abstract}When $k<n$, we study the coherent systems that come from a
BGN extension in which the quotient bundle is strictly semistable.
In this case we describe a stratification of the moduli space of
coherent systems. We also describe the strata as complements of
determinantal varieties and we prove that these are irreducible
and smooth. These descriptions allow us to compute the Hodge
polynomials of this moduli space in some cases. In particular, we give
explicit computations for the cases in which $(n,d,k)=(3,d,1)$ and $d$ is even, obtaining from them the usual Poincar\'e polynomials.
\end{abstract}



\section{Introduction and statement of results}


A coherent system of type $(n,d,k)$ on an algebraic curve $X$ of genus $g$ which is
smooth and projective, consists of a pair $(E,V)$ where $E$ is a
vector bundle on $X$ of rank $n$ and degree $d$ and $V$ is a
subspace of dimension $k$ of sections of $E$. Coherent systems
were introduced by J. Le Potier \cite{LeP1}, and N. Raghavendra
and P. A. Vishwanath \cite{RV}. The study of
coherent systems is interesting for various reasons. Coherent systems are related to the
Brill--Noether problem for higher rank (see \cite{BG}) and to gauge theory. Regarding the latter,
for instance one has that the $\alpha$-stability condition is equivalent to
the existence of solutions to a certain set of gauge theoretic
equations, one of which is essentially the vortex equation (see
\cite{BG2}). Coherent systems are also a generalisation of linear series on
algebraic curves.

For these objects there is a notion of stability that depends on a
real parameter $\alpha$. A coherent subsystem $(E',V')$ is a
subbundle $E'$ of $E$ together with a subspace of sections
$V'\subset H^0(X,E')\cap V$. One defines the $\alpha$-slope as
$\mu_{\alpha}(E,V)={d\over n}+\alpha{k\over n}$. The coherent
system is called $\alpha$-semistable (resp. $\alpha$-stable) if
the $\alpha$-slope of every coherent subsystem is less than or equal to (resp.
smaller than) the $\alpha$-slope of the coherent system.

Using the notion of $\alpha$-(semi)stability, A. King and
P. E. Newstead (see \cite{KN}) constructed a GIT quotient for these
objects. They proved that for fixed $n$, $d$, $k$ and
$\alpha$, there exists a projective scheme
$\widetilde{G}(\alpha ;n,d,k)$ which is a coarse moduli space of
$\alpha$-semistable coherent systems of type $(n,d,k)$. Let
$G(\alpha ;n,d,k)$ be the moduli space of $\alpha$-stable
coherent systems of the given type.

In recent years these moduli spaces have been broadly studied by
S.~B.~Bradlow, O.~Garc\'{\i}a-Prada, V.~Mercat, V.~Mu\~noz and
P.~E. Newstead (see \cite{BGMN}, \cite{BGMMN} and \cite{BGMMN2})
for genus greater than or equal to two, and by H. Lange and P.~E.
Newstead for genus zero and one (see \cite{LN1}, \cite{LN2},
\cite{LN3} and \cite{LN4}).

In this paper we deal with the cases in which $g\geq 2$, $k<n$ and
$\alpha$ ``large''. Under these hypotheses the moduli space
$G_L(n,d,k)$ of $\alpha $-stable coherent systems for ``large''
$\alpha$ is birationally equivalent to a Grassmannian fibration
over $\mathcal{M}(n-k,d)$ (see Proposition \ref{prop:restate}),
where $\mathcal{M}(n,d)$ denotes the moduli space of stable
bundles of rank $n$ and degree $d$ on $X$. This is given by the
observation that a coherent system $(E,V)$ of fixed type $(n,d,k)$
corresponds to a certain extension of the form (BGN extension, see
Definition \ref{BGN})
$$ 0\to \mathcal{O}^{\oplus k}\to E\to F\to 0. $$ This is used in
\cite{BGMMN} to obtain some information on the geometry and the
cohomology of these moduli spaces; in particular, some Betti numbers, fundamental groups and flip loci are
computed.

However, there is not a good enough geometric description of these
moduli spaces. The results in \cite{BGMMN} do not cover fully the cases
in which the coherent system comes from a BGN extension in which
the quotient bundle $F$ is strictly semistable. In this article we
study these cases giving a stratification of these moduli spaces
by looking at the quotient bundle $F$. We also study their Hodge
polynomials.

The layout of the paper is as follows. Section 2 is a review of the theory
described in \cite{L} of universal families of extensions. In
Section 3 we give a summary of the results and definitions about
coherent systems that can be found in \cite{BGMN} and
\cite{BGMMN}.

In Section 4 we study the BGN extensions and we give the
conditions that a BGN extension must satisfy in order to
contradict $\alpha$-stability (Theorem \ref{bgn}). In Section 5
we estimate the codimension of the variety of semistable vector
bundles such that the coherent system that they induce is not
$\alpha$-stable (Theorem \ref{codimension-bad-theorem}). In
Section 6 we study the sets that classify the quotient bundles
that appear in the BGN extensions associated to our coherent
systems. To do that, from the results in Section 4 we must look at
the Jordan--H\"older filtrations that are admitted by a given $F$.
Then, we study the possible sets of these filtrations and we give
geometric descriptions of them in terms of sequences of projective
fibrations (see Proposition \ref{sequenceof projective} for a
general construction). We also estimate the number of parameters on
which these sets depend. This description will allow us in
Section 7 to construct a stratification of the moduli space of
coherent systems for $n<k$ in some cases (Theorem \ref{Theo}). We
also describe these strata as complements of determinantal
varieties (Theorem \ref{deter}) and we prove that they
are smooth and irreducible (Theorem \ref{irre_smooth}). We finish
this paper studying the Hodge polynomials of these moduli spaces.
We start Section 8 by giving a review of Hodge theory and the
relationship between Hodge--Deligne and Hodge--Poincar\'e
polynomials that we denote by $\mathcal{H}$ and $HP$ respectively.
For a complex algebraic variety $X$, not necessarily smooth,
compact or irreducible, we define its \emph{Hodge--Deligne
 polynomial} (or virtual Hodge
polynomial) as
$$\mathcal{H}(X)(u,v)=\sum_{p,q}(-1)^{p+q}\chi_{p,q}^c(X)u^p v^q \in
\mathbb{Z}[u,v],$$and its \emph{Hodge--Poincar\'e polynomial} as
$$HP(X)(u,v)=\sum_{p,q}(-1)^{p+q}\chi_{p,q}(X)u^p v^q =\sum_{p,q,k}(-1)^{p+q+k}h^{p,q}(H^k(X))u^p v^q .$$
Here the Euler characteristics that we consider, $\chi_{p,q}^c$
and $\chi_{p,q}$ respectively, are the sums of the dimensions of
certain filtrations associated to the cohomology groups with
compact support and to the usual cohomology groups for the
Hodge--Deligne polynomials and the Hodge--Poincar\'e polynomials
respectively. We also introduce equivariant Hodge--Poincar\'e
polynomials and we study how to compute the Hodge--Poincar\'e
polynomials of the strata in a general setup (see Theorem
\ref{HodgeStratum}). We conclude the paper by giving explicit
computations of some cases in which $n-k=2$. These are the
following. The Hodge--Deligne polynomial of the moduli space
$G_L(n,d,k)$ for $n-k=2$ and $d$ odd is (see Theorem
\ref{H1})
\begin{align*}\mathcal{H}(G_L(n,d,k))(u,v)
=&(1+u)^g(1+v)^g\cdot
\frac{(1+u^2v)^g(1+uv^2)^g-u^gv^g(1+u)^g(1+v)^g}{(1-uv)(1-u^2v^2)}\cdot
\\& \cdot \frac{(1-(uv)^{2(g-1)+d-k+1})\cdot \ldots \cdot
(1-(uv)^{2(g-1)+d})}{(1-uv)\cdot \ldots \cdot
(1-(uv)^{k})}.\end{align*}In
Theorem \ref{H2} we compute the Hodge--Deligne polynomial of $G_L(3,d,1)$ when $d$ is even and $g\geq (3-d)/2$. From the latter theorem one can obtain the usual Poincar\'e polynomial of $G_L(3,d,1)$ just by writing $u=v=t$, this is given by (see Corollary \ref{ya}):\begin{align*}P&_{G_L(3,d,1)}(t)  =\\& =\frac{(1+t)^{2 g}(-t^2 +t^{d+2g})}{{t^6} {{(-1+{t^2})}^3} (1+{t^2})}\Big({-t}^{6+2g}(1+t)^{2g} +
t^4(1+t^3)^{2g}-{t}^{4g+d}(1+t)^{2g}+t^{2+d+2g}(1+t^3)^{2g}\Big).
\end{align*}

\section{Universal families of extensions}\label{Lange}


Here we introduce some theory of universal families of extensions,
the conditions for the existence of global universal families are
given as well as the conditions for the existence of universal
families in a ``local'' sense. All these results can be found in
\cite{L}.

Let $f:X\rightarrow Y$ be a flat projective morphism of noetherian
schemes and $\mathscr{F}$ and $\mathscr{G}$ coherent
$\mathcal{O}_X$-modules, flat over $Y$. Let
$Ext^1_X(\mathscr{F},\mathscr{G})$ be the vector space
parametrizing the extensions of $\mathscr{F}$ by $\mathscr{G}$
over $X$. Let $\mathscr{E}xt_f^i(\mathscr{F},\mathscr{G}):=
R^i(f_{\ast}\mathscr{H}om_{\mathcal{O}_X}(\mathscr{F},\bullet
))(\mathscr{G})$ be the ith relative Ext-sheaf.

We restrict ourselves to the case in which $\mathscr{F}$ is
locally free. In this case\begin{equation*}\mathscr{E}xt_f^i(\mathscr{F},\mathscr{G})\simeq
R^i f_{\ast} (\mathscr{F}^{\vee}\otimes
\mathscr{G}),\end{equation*}and for every coherent sheaf $\mathscr{C}$
on X and for every point $y\in Y$, the usual base change
homomorphism
$$\tau^i(y):R^i f_{\ast} \mathscr{C}\otimes k(y)\rightarrow H^i(X_y
,\mathscr{C}_y)$$is the homomorphism
\begin{equation}\varphi^i(y):\mathscr{E}xt_f^i(\mathscr{F},\mathscr{G})\otimes
k(y)\rightarrow
Ext_{X_y}^i(\mathscr{F}_y,\mathscr{G}_y).\end{equation}

We will define now what a family of extensions is. For every point
$y\in Y$ let
\begin{equation*}\phi_y
:Ext_X^1(\mathscr{F},\mathscr{G})\rightarrow
Ext_{X_y}^1(\mathscr{F}_y ,\mathscr{G}_y),
\end{equation*}be the map that assigns to every extension class of
$\mathscr{F}$ by $\mathscr{G}$, the extension class of
$\mathscr{F}_y$ by $\mathscr{G}_y$.
\begin{definition}\textnormal{A \emph{family of extensions} of
$\mathscr{F}$ by $\mathscr{G}$ over $Y$ is a family $(e_y \in
Ext_{X_y}^1(\mathscr{F}_y ,\mathscr{G}_y))_{y\in Y}$ such that
there is an open covering $\mathscr{U} =(U_i)_{i\in I}$ of $Y$ and
for each $i\in I$ an element $\sigma_i \in
Ext^1_{f^{-1}(U_i)}(\mathscr{F}_{| f^{-1}(U_i)},\mathscr{G}_{|
f^{-1}(U_i)})$ such that $e_y = \phi_{i,y}(\sigma_i)$ for every
$y\in Y$. Here $\phi_{i,y}$ denotes the canonical map
$$Ext^1_{f^{-1}(U_i)}(\mathscr{F}_{| f^{-1}(U_i)},\mathscr{G}_{|
f^{-1}(U_i)})\rightarrow Ext_{X_y}^1(\mathscr{F}_y
,\mathscr{G}_y).$$The family of extensions is called
\emph{globally defined} if the covering $\mathscr{U}$ may be taken
to be $Y$ itself.}
\end{definition}

The relationship between the groups $\mathscr{E}xt^i_f(\mathscr{F},\mathscr{G})$ and $Ext^{j}_X(\mathscr{F},\mathscr{G})$, is accounted for by a spectral sequence whose $E_2$-term is given by
$E_2^{p,q}=H^p(Y,\mathscr{E}xt^q_f(\mathscr{F},\mathscr{G}))$ and
which abuts to $Ext^{\ast}_X(\mathscr{F},\mathscr{G})$.

Suppose in addition to the general hypotheses that
$\mathscr{E}xt^i_f (\mathscr{F},\mathscr{G})$ commutes with base
change for $i=0$, $1$. Let $q_S:X\times_Y S \rightarrow X$ and
$p_S :X\times_Y S \rightarrow S$ be the projections. Then the
functor
$$E(S):=H^0(S,\mathscr{E}xt^1_{p_S}(q_S^{\ast}\mathscr{F},q_S^{\ast}\mathscr{G}))$$of
the category of noetherian $Y$-schemes to the category of sets, is
a contravariant functor that is representable by the vector bundle
$V=\mathbb{V}(\mathscr{E}xt^1_f (\mathscr{F},\mathscr{G})^{\vee})$
over $Y$ associated to the locally free sheaf $\mathscr{E}xt^1_f
(\mathscr{F},\mathscr{G})^{\vee}$.
\begin{proposition}\label{Ext-first}Suppose $Y$ is reduced and $\mathscr{E}xt^i_f
(\mathscr{F},\mathscr{G})$ commutes with base change for $i=0$,
$1$. Then there is a family $(e_v)_{v\in V}$ of extensions of
$q_V^{\ast}\mathscr{F}$ by $q_V^{\ast}\mathscr{G}$ over $V$ which
is universal in the category of reduced noetherian $Y$-schemes.
\end{proposition}

\begin{parrafo}\label{universal}\textnormal{Here ``universal'' means: Given a reduced noetherian
$Y$-scheme $S$ and a family of extensions $(e_s)_{s\in S}$ of
$q_S^{\ast}\mathscr{F}$ by $q_S^{\ast}\mathscr{G}$ over $S$, then
there is exactly one morphism $g:S\rightarrow V$ over $Y$ such
that $(e_s)_{s\in S}$ is the pull-back of $(e_v)_{v\in V}$ by
$g$.}\end{parrafo}

There exists a projective analogue of the above result. Under the
same hypotheses as in the last proposition, consider the functor
$$PE(S):=\textnormal{set of invertible quotients of
$\mathscr{E}xt^1_{p_S}(q_S^{\ast}\mathscr{F},q_S^{\ast}\mathscr{G})^{\vee}$}$$of
the category of noetherian $Y$-schemes to the category of sets,
where $q_S$ and $p_S$ are as above. This is a contravariant
functor that is representable by the projective bundle
$P=\mathbb{P}(\mathscr{E}xt^1_f (\mathscr{F},\mathscr{G})^{\vee})$
over $Y$.
\begin{proposition}\label{ext-projective}Suppose $Y$ is reduced and $\mathscr{E}xt^i_f
(\mathscr{F},\mathscr{G})$ commutes with base change for $i=0$,
$1$. Then there is a family $(e_p)_{p\in P}$ of extensions of
$q_P^{\ast}\mathscr{F}$ by $q_P^{\ast}\mathscr{G} \otimes
p_P^{\ast}\mathcal{O}_P(1)$ over $P$ which is universal in the
category of reduced noetherian $Y$-schemes for the classes of
families of non-split extensions of $q_P^{\ast}\mathscr{F}$ by
$q_P^{\ast}\mathscr{G} \otimes p_P^{\ast}\mathscr{L}$ over $S$
with arbitrary $\mathscr{L}\in Pic(S)$ modulo the canonical
operation of $H^0(S,\mathcal{O}_S^{\vee})$.
\end{proposition}

As a restriction of these results we obtain the
classical ones on universal extensions, here ``universal'' is in
the usual sense. These are the following (see \cite{NR}, \cite{R}
and the Appendix on extensions of \cite{S}):

Fix an algebraic variety $X$, and let $S$ and $T$ be two more
algebraic varieties. Let $V$ (resp. $W$) be a vector bundle on
$S\times X$ (resp. $T\times X$), such that
dim($H^1(X,\mathscr{H}om (W_t , V_s))$) is independent of the
point $(s,t)$ of $S\times T$. Let $p_{S\times T}$, $p_T$ and $p_S$
be the projections $S\times T \times X \rightarrow S \times T$,
$S\times T \rightarrow T$ and $S\times T \rightarrow S$
respectively.

Let $$F=R^1(p_{S \times T})_{\ast}(\mathscr{H}om ((p_T \times
id_X)^{\ast} W, (p_S \times id_X)^{\ast} V)).$$This is a vector
bundle on $S\times T$. Let $\pi :F \rightarrow S \times T$ be the
projection.
\begin{proposition}\label{Ext-NR}If $$h^i(S\times T, (p_{S \times T})_{\ast}(\mathscr{H}om ((p_T \times
id_X)^{\ast} W, (p_S \times id_X)^{\ast} V)\otimes F^{\vee})=0$$
for $i=1$, $2$, there exists a vector bundle $E$ on $F\times X$
and an exact sequence
\begin{equation*}0\rightarrow (\pi \times
id_X)^{\ast}(p_S \times id_X)^{\ast} V \rightarrow E \rightarrow
(\pi \times id_X)^{\ast}(p_T \times id_X)^{\ast} W \rightarrow 0,
\end{equation*}such that for every point $(s,t)\in S\times T$ and
every element $h \in F_{(s,t)}=H^1(X ,\mathscr{H}om (W_t ,V_s))$,
its restriction to $\{ h \} \times
X$:\begin{equation*}0\rightarrow V_s \rightarrow E_h \rightarrow
 W_t \rightarrow 0
\end{equation*}is the extension associated to $h$.
\end{proposition}

As in the general case, we have a projective analogue of this
proposition.

\begin{remark}\label{remark}{\em The hypotheses of Proposition \ref{Ext-NR} are
verified in the following cases:
\begin{itemize}\item[(a)] When for all $(s,t)\in S \times T$, we have
that $Hom (W_t ,V_s)=\{ 0 \}$. \item[(b)] When $S$ and $T$ are
affine.
\end{itemize}}
\end{remark}


\section{Coherent systems}\label{Coherent}


In this section we introduce some general theory on coherent
systems on algebraic curves. This material is a summary of
results that can be found in \cite{BG}, \cite{BGMN} and \cite{BGMMN}.

Let $X$ be a smooth projective algebraic curve of genus greater
than or equal to 2.
\begin{definition}\textnormal{A \emph{coherent system} on $X$ of type
$(n,d,k)$ is a pair $(E,V)$, where $E$ is a vector bundle on $X$
of rank $n$ and degree $d$ and $V$ is a subspace of dimension $k$
of the space of sections $H^0(E)$.}
\end{definition}
\begin{definition}\textnormal{Fix $\alpha \in \mathbb{R}$. Let $(E,V)$ be a
coherent system of type $(n,d,k)$. The \emph{$\alpha$-slope} of
$(E,V)$, $\mu_{\alpha}(E,V)$, is defined by $$\mu_{\alpha}
(E,V)=\frac{d}{n}+\alpha \cdot \frac{k}{n}.$$We say that $(E,V)$
is \emph{$\alpha$-stable} if
$$\mu_{\alpha}(E',V')<\mu_{\alpha}(E,V)$$for all proper subsystems
$(E',V')$ (i.e. for every non-zero subbundle $E'$ of $E$ and
every subspace $V' \subseteq V \cap H^0(E')$ with $(E',V')\neq
(E,V)$). Analogously \emph{$\alpha$-semistability} is defined by
changing $<$ to $\leqslant$.}
\end{definition}

There exists a (coarse) moduli space for $\alpha$-stable coherent
systems of type $(n,d,k)$ which we denote by $G(\alpha ;n,d,k)$.

\begin{definition}\textnormal{We say that $\alpha >0$ is a \emph{critical value} if
there exists a proper subsystem $(E',V')$ such that
$\frac{k'}{n'}\neq \frac{k}{n}$ but
$\mu_{\alpha}(E',V')=\mu_{\alpha}(E,V)$. We also regard $0$ as a
critical value.}
\end{definition}

For $\alpha$ not critical, if $\gcd(n,d,k)=1$, the
$\alpha$-semistability condition and the $\alpha$-stability
condition are equivalent. For $k<n$, it is easy to see that there
are finitely many critical values. This is also true, but not obvious, when $k\geq
n$.

If we label the critical values of $\alpha$ by $\alpha_i$,
starting with $\alpha_0 =0$, we get a partition of the
$\alpha$-range into a set of intervals $(\alpha_i
,\alpha_{i+1})$. Within the interval $(\alpha_i ,\alpha_{i+1})$
the property of $\alpha$-stability is independent of $\alpha$,
that is if $\alpha$, $\alpha' \in (\alpha_i , \alpha_{i+1})$ then
$G(\alpha ;n,d,k)=G(\alpha';n,d,k)$. We shall denote this moduli
space by $G_i$.

Suppose now that $G(\alpha ;n,d,k)\neq \emptyset$ for at least one
value of $\alpha$.
\begin{proposition}Let $k<n$ and let $\alpha_L$ be the biggest critical value
smaller than $\frac{d}{n-k}$. The $\alpha$-range is divided into
a finite set of intervals determined by critical values
$$0=\alpha_0 <\alpha_1 <\alpha_2 <\ldots <\alpha_L
<\frac{d}{n-k}.$$If $\alpha
>\frac{d}{n-k}$, the moduli spaces are empty.
\end{proposition}

The difference between adjacent moduli spaces in the family $G_0$,
$G_1$, ..., $G_L$ is accounted for by the subschemes $G_i^+
\subseteq G_i$ and $G_i^- \subseteq G_{i-1}$, where $G_i^+$
consists of all $(E,V)$ in $G_i$ which are not $\alpha$-stable if
$\alpha < \alpha_i$ and $G_i^- \subseteq G_{i-1}$ contains all
$(E,V)$ in $G_{i-1}$ which are not $\alpha$-stable if $\alpha >
\alpha_i$. It follows that $G_i -G_i^+=G_{i-1}-G_i^-$ and that $G_i$ is transformed into $G_{i-1}$ by removal
of $G_i^+$ and the insertion of $G_i^-$.

\begin{definition}\textnormal{We refer to such a procedure as a \emph{flip}. We call the subschemes $G_i^{\pm}$ \emph{the flip loci}. We say
that a flip is \emph{good} if the flip loci have strictly positive codimension in
every component of the moduli spaces $G_i$ and $G_{i-1}$
respectively. Under these conditions the moduli spaces are
birationally equivalent.}
\end{definition}

\section{Study of the BGN extensions}


When $k<n$ we denote by $G_L(n,d,k)$ the moduli space of coherent
systems of type $(n,d,k)$ for $\alpha$ large, i.e., $\alpha_L
<\alpha < \frac{d}{n-k}$.

\begin{definition}[\cite{BG,BGN}]\textnormal{\label{BGN}A \emph{BGN extension} is an extension of
vector bundles
\begin{equation*}0 \rightarrow \mathcal{O}^{\oplus k}\rightarrow E
\rightarrow F \rightarrow 0
\end{equation*}which satisfies the following conditions:
\begin{itemize}\item[(i)] rank$E=n>k$,
\item[(ii)] deg$E=d>0$, \item[(iii)] $H^0(F^{\ast})=0$,
\item[(iv)] if $e=(e_1,\ldots ,e_k)\in H^1(F^{\ast}\otimes
\mathcal{O}^{\oplus k})=H^1(F^{\ast})^{\oplus k}$ denotes the
class of the extension, then $e_1 ,\ldots ,e_k$ are linearly
independent as vectors in $H^1(F^{\ast})$.\end{itemize}}
\end{definition}

\begin{definition}\textnormal{Two BGN extensions are \emph{equivalent} if one has a
commutative diagram
\begin{equation*}\xymatrix@R=0.5cm{0 \ar[r]& \mathcal{O}^{\oplus k'}\ar[d]\ar[r]& E'
\ar[r]\ar[d]& F' \ar[r]\ar[d]& 0 \\
0 \ar[r]& \mathcal{O}^{\oplus k}\ar[r]& E \ar[r]& F \ar[r]& 0}
\end{equation*}where the vertical arrows are isomorphisms, in particular $k=k'$. An
equivalence class of BGN extensions will be called a \emph{BGN
extension class}.}
\end{definition}

\begin{proposition}[\cite{BG},\cite{BGMN}]\label{bgmn}Suppose that $0<k<n$ and $d>0$. Let
$\alpha_L <\alpha <\frac{d}{n-k}$. Let $(E,V)$ be an
$\alpha$-semistable coherent system of type $(n,d,k)$. Then
$(E,V)$ defines a BGN extension class represented by an extension
\begin{equation*}0 \rightarrow \mathcal{O}^{\oplus k}\rightarrow E
\rightarrow F \rightarrow 0
\end{equation*}with $F$ semistable. Conversely, any BGN extension
of type $(n,d,k)$ in which the quotient $F$ is stable gives rise
to an $\alpha$-stable coherent system of the same type.
\end{proposition}

\begin{remark}\textnormal{Note that if in the last part of Proposition
\ref{bgmn} our quotient bundle $F$ is only semistable, the
coherent system can fail to be $\alpha$-stable or even
$\alpha$-semistable.}
\end{remark}

\begin{proposition}[\cite{BG},\cite{BGMN},\cite{BGMMN}]\label{prop:restate}
Suppose $n\ge2$ and $0<k\le n$. Then $G_L(n,d,k)\ne\emptyset$ if
and only if $$\textnormal{$d>0$, $ $ $k\le n+\frac1g(d-n)$ $ $ and
$ $ $(n,d,k)\ne(n,n,n)$,}$$and it is then always irreducible and
smooth of dimension $\beta(n,d,k)=n^2(g-1)+1-k(k-d+n(g-1))$.

If $0<k<n$, $G_L(n,d,k)$ is birationally equivalent to a fibration
over the moduli space of stable vector bundles, $\mathcal{M}(n-k,d)$ with fibre the
Grassmannian $\Gr(k,d+(n-k)(g-1))$. More precisely, if $W$ denotes
the subvariety of $G_L(n,d,k)$ consisting of coherent systems for
which the quotient bundle $F$ is strictly semistable, then
$G_L(n,d,k)\setminus W$ is isomorphic to a Grassmann fibration
over $\mathcal{M}(n-k,d)$.

If in addition $\gcd(n-k,d)=1$, then $W=\emptyset$ and
$G_L(n,d,k)\rightarrow \mathcal{M}(n-k,d)$ is the Grassmann
fibration associated to some vector bundle over
$\mathcal{M}(n-k,d)$.
\end{proposition}

Our next goal is to study what happens when the quotient bundle $F$ is
strictly semistable. To this end, consider a BGN extension as
above:
\begin{equation}\label{extension1}0 \rightarrow \mathcal{O}^{\oplus k}\rightarrow E
\rightarrow F \rightarrow 0,
\end{equation}
in which $F$ is strictly semistable with rank $n-k$ and degree
$d> 0$. Let
\begin{equation*}(e_1 ,\ldots ,e_k)\in H^1(F^{\ast}\otimes \mathcal{O}^{\oplus
k})=H^1(F^{\ast})^{\oplus k}
\end{equation*}be the class of the extension of $E$, with $\{ e_i \}_i$ linearly independent
as vectors in $H^1(F^{\ast})$.

Let $(E,V)$ be the coherent system corresponding to the extension
(\ref{extension1}). Consider a subsystem $(E',V')$. In general,
this subsystem determines an extension
\begin{equation*}\xymatrix@R=0.5cm{0 \ar[r]& W'\ar[d]\ar[r]& E'
\ar[r]\ar[d]& F' \ar[r]\ar@{^{(}->}[d]& 0 \\
0 \ar[r]& \mathcal{O}^{\oplus k}\ar[r]& E \ar[r]& F \ar[r]& 0}
\end{equation*}with $F'$ a subsheaf of $F$, $W'$ a subbundle of
$\mathcal{O}^{\oplus k}$. Let $\alpha \in (\alpha_L ,
\frac{d}{n-k})$ be sufficiently close to $ \frac{d}{n-k}$. We are
going to study the relationship between $\mu_{\alpha}(E',V')$ and
$\mu_{\alpha}(E,V)$. It is proved in \cite{BG} (Lemma 4.3) that for an extension
\begin{equation*}0\rightarrow W'\rightarrow E'
\rightarrow F' \rightarrow 0,
\end{equation*}$\deg (W')\leq 0$ and it is equal to $0$ if and only if $W'\cong
\mathcal{O}^{\oplus k'}$. Moreover
$h^0(W')\leq \rank (W')$ and it is equal if and only if $W'\cong
\mathcal{O}^{\oplus k'}$. This can be proved by considering the
vector bundle generated by the global sections of $W'$ and
bearing in mind that if $h^0(W')= \rank (W')$ and $W'\ncong
\mathcal{O}^{\oplus k'}$ then the degree of $W'$ would be
positive, which contradicts the fact that $\deg (W')\leq 0$.

We divide the study into the following cases:

\subsection*{\S . $F'$ proper, non-trivial subsheaf and $W'\ncong
\mathcal{O}^{\oplus k'}$}

Let $l'=\rank (W')$ and $k'=h^0(W')$. Since $\deg(W')\leq 0$ we
have
\begin{align*}\mu (E')=\frac{n'-l'}{n'}\cdot\frac{(\deg W' +\deg F')}{n'-l'}\leq
\frac{\deg F'}{n'-l'}\cdot \frac{n'-l'}{n'}\leq \mu (F) \cdot
\frac{n'-l'}{n'}.
\end{align*}

Following the computations of \cite{BG} page 139, we have
\begin{equation*}\mu_{\alpha}(E',V')-\mu_{\alpha}(E,V)\leq \frac{\varepsilon}{n-k}
\bigg{[} \frac{k}{n}-
\frac{k'}{n'}\bigg{]}+\mu(F)\frac{k'-l'}{n'},
\end{equation*}where $\varepsilon =d-\alpha (n-k) >0$ and
we know that $d>0$ and $\mu(F)>0$. Lemma 4.3 of \cite{BG} implies
that $k'< l'$, so choosing $\varepsilon$ sufficiently small -note
that in Section \ref{Coherent} we saw that the condition of
$\alpha$-stability does not change within an interval $(\alpha_i
,\alpha_{i+1})$, i.e., $G(\alpha ;n,d,k)=G(\alpha' ;n,d,k)$ for
all $\alpha ,\alpha' \in (\alpha_i ,\alpha_{i+1})$- we have
\begin{equation*}\frac{\varepsilon}{n-k}
\bigg{[} \frac{k}{n}-
\frac{k'}{n'}\bigg{]}+\mu(F)\frac{k'-l'}{n'}<0\end{equation*}and
we conclude that $(E',V')$ does not contradict the
$\alpha$-stability of $(E,V)$.

\subsection*{\S . The cases: $F'=0$; $F'=F$ and $W'\ncong
\mathcal{O}^{\oplus k'}$; $F'=F$ and $W'=\mathcal{O}^{\oplus
k'}$.}

These cases follow in the same way as Theorem 4.2 of \cite{BG},
and $(E',V')$ does not contradict the $\alpha$-stability of
$(E,V)$.

\subsection*{\S . $F'$ a proper, non-trivial subsheaf and
$W'=\mathcal{O}^{\oplus k'}$}

We know that $\deg (W')=0$, and $(E',V')$ is a subsystem of type
$(n',d',k')$ where:
\begin{equation*}\textnormal{$d'=\deg F'$, $n'=k'+\rank F'$, and
$\mu (F')\leq \mu (F)$ $\Rightarrow$ $\frac{d'}{n'-k'}\leq
\frac{d}{n-k}$}.
\end{equation*}
So, bearing in mind that $\alpha \in (\alpha_L , \frac{d}{n-k})$
sufficiently close to $ \frac{d}{n-k}$, let $\alpha =\frac{d}{n-k}
- \frac{\varepsilon}{n-k}$ with $\varepsilon$ sufficiently small,
we have
\begin{align}\label{stable}\mu_{\alpha}(E',V')& \nonumber -\mu_{\alpha}(E,V)
=\frac{d'}{n'}-\frac{d}{n}+\alpha ( \frac{k'}{n'}-
\frac{k}{n})=\\&
=\frac{1}{n'(n-k)}[nd'-dn'-(kd'-dk')]+\frac{\varepsilon}{n-k}
\bigg{[} \frac{k}{n}- \frac{k'}{n'}\bigg{]}.
\end{align}In the case in which $\frac{d}{n-k}=\mu (F)>\mu
(F')=\frac{d'}{n'-k'}$, we have $nd'-dn'-(kd'-dk')<0$, so choosing
$\varepsilon$ properly, this does not contradict the
$\alpha$-stability of $(E,V)$. In the other case, when $\mu
(F)=\mu (F')$, we obtain
\begin{align*}\mu_{\alpha}(E',V') \nonumber -\mu_{\alpha}(E,V)=
\frac{\varepsilon}{n-k} \bigg{[} \frac{k}{n}-
\frac{k'}{n'}\bigg{]}.
\end{align*}So the $\alpha$-(semi)stability depends on whether $\frac{k}{n}-
\frac{k'}{n'}$ is greater, equal or less than $0$. Hence, we only have
trouble when $\frac{k}{n}\geq \frac{k'}{n'}$.

If we study the relationship between the invariants $k$, $k'$, $n$
and $n'$, we find that
\begin{equation*}\rank(F')=n'-k'<\rank(F)=n-k,
\end{equation*}because $F'$ is a proper subbundle
of $F$.

From now on, we will study the cases in which the coherent system
that comes from our original BGN extension fails to be
$\alpha$-stable. These cases are those in which we have a
coherent subsystem $(E',V')$ of type $(n',d',k')$ such that
$n-k>n'-k'>0$, $\frac{k}{n}\geq \frac{k'}{n'}$ and
$\mu(F)=\mu(F')$. Note that the condition $\frac{k}{n}\geq
\frac{k'}{n'}$ is equivalent to $\frac{k}{n-k}\geq
\frac{k'}{n'-k'}$. We can restrict ourselves to the case in which
our extension
\begin{equation}\label{BGNE''}\xymatrix@C=0.4cm{0 \ar[r] &\mathcal{O}^{\oplus
k'}\ar[r]& E' \ar[r]& F' \ar[r]& 0 }
\end{equation}is either an extension verifying the properties (i)-(iii)
of the definition of BGN extension (Definition \ref{BGN}) or, for
the smallest value of $k'$ for which the extension (\ref{BGNE''})
exists, a BGN extension. This is proved in
\begin{theo}\label{bgn}A BGN extension (\ref{extension1}) fails to be $\alpha$-stable
for $\alpha_L < \alpha < \frac{d}{n-k}$ if and only if there
exists a BGN extension (\ref{BGNE''}) and a commutative diagram
\begin{equation}\label{diagram}\xymatrix@R=0.4cm{& 0
\ar[d]& 0\ar[d]& 0 \ar[d]\\
0 \ar[r] &\mathcal{O}^{\oplus
k'}\ar[d]\ar[r]& E' \ar[d]\ar[r]& F' \ar[d]\ar[r]& 0 \\
0 \ar[r] &\mathcal{O}^{\oplus k}\ar[r]& E \ar[r]& F \ar[r]& 0  }
\end{equation}such that
\begin{equation}\label{conditions}\textnormal{$n-k>\rank F' >0$, $\frac{k}{n-k}\geqslant \frac{k'}{\rank F'}$
and $\mu (F) =\mu (F')$.}
\end{equation}
\end{theo}
\begin{proof}The existence of (\ref{diagram}) immediately implies
that (\ref{extension1}) is not $\alpha$-stable. Conversely, if
(\ref{extension1}) is not $\alpha$-stable, then there exists
a diagram (\ref{diagram}) for which (\ref{conditions}) holds. We
need only show that we can choose (\ref{diagram}) so that
(\ref{BGNE''}) is a BGN extension.

Now (\ref{diagram}) and (\ref{conditions}) immediately imply
conditions (i) and (ii) of Definition \ref{BGN} for the extension
(\ref{BGNE''}). Moreover, since $F$ is semistable, so is $F'$;
since $\deg F'
>0$ this implies that $H^0(F'^{\ast})=0$, giving (iii).

Condition (iv), however, is not automatic. Let $(e_1,\ldots ,e_k)$
be the $k$-tuple classifying (\ref{extension1}) and let
$(e'_1,\ldots ,e'_k)$ be the image of this $k$-tuple under the
surjection
$H^1(F^{\ast})\xymatrix@C=0.5cm{\ar@{->>}[r]&}H^1(F'^{\ast})$. Put
$k''=\dim <e'_1,\ldots ,e'_k>$. The existence of (\ref{diagram})
implies that $k'\geqslant k''$. On the other hand, after applying
an automorphism of $\mathcal{O}^{\oplus k}$, we can assume that
$e'_{k''+1}=\ldots =e_{k}'=0$ and hence deduce the existence of a
subextension of (\ref{extension1}) of the form
\begin{equation}\label{extensionkdos}0\rightarrow \mathcal{O}^{\oplus k''} \rightarrow E''
\rightarrow F' \rightarrow 0
\end{equation}classified by $(e'_1,\ldots ,e'_{k''})$. This is a
BGN extension.

Moreover, bearing in mind (\ref{conditions}) and the fact that $k' \geqslant k''$, one has that $\frac{k}{n-k}\geqslant \frac{k''}{\rank F'}$ and the
extension (\ref{extensionkdos}) satisfies all the conditions
required for (\ref{BGNE''}).
\end{proof}


\section{The codimension of the ``bad'' part}\label{codimension-bad}


We estimate the codimension of the
subvariety of $\oplus^k H^1(F^{\ast})$ whose elements are ``bad'',
in the sense that the coherent systems they induce are not
$\alpha$-stable. We call this subvariety $S$.

\begin{theo}\label{codimension-bad-theorem}Suppose that $F$ has only finitely many subbundles $F'$ with
$\mu (F)= \mu (F')$. Then, the co-dimension of the subvariety $S$
of $H^1(F^{\ast})^{\oplus k}$ satisfies
\begin{equation}\label{dimen-estimates}\codim (S)\geq \min \{ (
(g-1)n'-k'g+d')(k-k') \},
\end{equation}where this minimum is taken over all the invariants
$n'$, $d'$, $k'$ for which $F$ possesses a subbundle $F'$ of type $(n'-k',d')$ satisfying
\begin{equation}\label{conditionsdos}\textnormal{$n-k>n'-k' >0$, $\frac{k}{n}\geq \frac{k'}{n'}$
and $\frac{d}{n-k}=\frac{d'}{n'-k'}$.}
\end{equation}
\end{theo}\begin{proof}By Theorem \ref{bgn} the BGN extensions that could give us any trouble are
those which possess BGN subextensions of the form (\ref{BGNE''}).
Bearing this in mind, to calculate the codimension we consider the
following picture
\begin{equation}\label{dimension}\xymatrix@R=0.5cm{0 \ar[r] &\mathcal{O}^{\oplus
k'}\ar@{^{(}->}[d]\ar[r]& E' \ar[d]\ar[r]& F' \ar@{=}[d]\ar[r]& 0 \\
0 \ar[r] &\mathcal{O}^{\oplus k}\ar@{=}[d]\ar[r]& E_1
\ar[d]\ar[r]&
F' \ar@{^{(}->}^g [d]\ar[r]& 0 \\
0 \ar[r] &\mathcal{O}^{\oplus k}\ar[r]& E \ar[r]& F \ar[r]& 0}
\end{equation}where the first horizontal extension is a BGN subextension
of our original BGN extension, which is the bottom one. We call
now $(g^{\ast}e_1, \ldots ,g^{\ast}e_k)$ the $k$-tuple image of
the $k$-tuple $(e_1, \ldots ,e_k)$ under the map
$H^1(F^{\ast})\xymatrix@C=0.5cm{\ar@{->>}[r]&}H^1(F'^{\ast})$
induced by the canonical immersion $g$. The existence of the first
horizontal extension tells us precisely that at most $k'$ elements
of $(g^{\ast}e_1, \ldots ,g^{\ast}e_k)$ are linearly independent.
Using Riemann-Roch and bearing in mind that $H^0(F'^{\ast})=0$,
we get
\begin{equation*}h^1(F'^{\ast})=(g-1)(n'-k')+d';
\end{equation*}this identity tells us that the codimension of
the subvariety $S_{F',k'}$ of $\oplus^k H^1(F^{\ast})$, where the
subindex $F'$ refers to the subbundle $F'$ of $F$, satisfies
\begin{align}\label{bad}\codim(S_{F',k'})&\geq \big{(}h^1(F'^{\ast})-k'\big{)}\big{(}k-k'\big{)}=\\&
=\big{(}(g-1)n'-k'g+d'\big{)}\big{(}k-k'\big{)}.\nonumber
\end{align}

So, if we look at all the subbundles $F'$ of $F$ for which $\mu
(F')= \mu (F)$, we see the codimension of the
subvariety of $H^1(F^{\ast})^{\oplus k}$ of ``bad'' elements
satisfies (\ref{dimen-estimates}), where the minimum is taken over
all the invariants $n'$, $d'$, $k'$ satisfying
(\ref{conditionsdos}).
\end{proof}

\begin{remark}{\em Condition (iv) of Definition \ref{BGN}
tells us that $$k\leq h^1(F^{\ast})=d+(n-k)(g-1).$$From this and
(\ref{conditionsdos}) \begin{align*}k'(n-k)\leq k(n'-k')& \leq
d(n'-k')+(n-k)(n'-k')(g-1)=\\& =d'(n-k)+ (n-k)(n'-k')(g-1)=\\&
=(n-k)h^1(F'^{\ast}).\end{align*}So $k'\leq h^1(F'^{\ast})$,
proving that there exist BGN extensions (\ref{BGNE''}) and also
that the lower bound for $\codim S$ is greater than or equal to 0.

Moreover, $k'=h^1(F'^{\ast})$ is only possible if the above
inequalities are equalities. In particular
$\frac{k'}{n'}=\frac{k}{n}$, hence also
$\frac{d'}{n'}=\frac{d}{n}$, and $k=h^1(F^{\ast})$. Writing
$\lambda =k/n$ and $\mu =d/n$, this means that
\begin{equation}\label{Straight}\lambda = 1+ \frac{1}{g}(\mu
-1)\end{equation}(see \cite{BGN} and \cite{Me}). So
(\ref{extension1}) and (\ref{BGNE''}) correspond to the same point
in the Brill--Noether map of \cite{BGN} and this point lies on the
line given by (\ref{Straight}).

Conversely, if the point corresponding to (\ref{extension1}) lies
on the line (\ref{Straight}) and $\gcd(n,d,k)>1$, we can find
$(n',d',k')$ with $\frac{k'}{n'}=\frac{k}{n}$,
$\frac{d'}{n'}=\frac{d}{n}$ and $n'-k' <n-k$.

If this happens and $F$ possesses a subbundle $F'$ with invariants
$(n'-k',d')$, then $h^1(F'^{\ast})=k'$ and the diagram
(\ref{diagram}) exists, proving that the corresponding $(E,V)$ is
not $\alpha$-stable. In this case there are no $\alpha$-stable
$(E,V)$ with quotient $F$.

In all other cases, the general $(E,V)$ is $\alpha$-stable.}
\end{remark}


\section{Some geometry of the spaces that classify the quotients}\label{subsection-epsilon}


In Theorem \ref{bgn} we saw that in order to find out if a
coherent system is not $\alpha$-stable we have to look at the
quotient bundle that appears in its associated BGN extension.
Those coherent systems that fail to be $\alpha$-stable satisfy
that their quotient bundle has subbundles with the same slope as
the quotient bundle and that satisfy the properties described in
the Theorem \ref{bgn}.

For a given vector bundle $F$, all the subbundles of $F$ whose
slope is the same as the slope of $F$, appear in some of the
Jordan--H\"older filtrations of $F$. Bearing this in mind, in this
section we study the sets of all possible Jordan--H\"older
filtrations of a given vector bundle. From these
sets we will define a stratification of $G_L(n,d,k)$.

We give some sort of ``universal'' constructions for these sets of
Jordan--H\"older filtrations, some of them are described as
projective fibrations, others are described in terms of ``local''
and global extensions, following the results and terminology of
Section \ref{Lange}. All these geometrical descriptions will
allow us in the following sections to describe our strata as
complements of determinantal varieties and prove irreducibility
and smoothness conditions for the strata.

First of all we need to introduce some definitions.
\begin{definition}\textnormal{A \emph{Jordan--H\"{o}lder filtration} of
length $r$ of a semistable vector bundle $F$ is a filtration
\begin{equation}\label{j-h-filtration}
0=F_0\subset F_1\subset F_2 \subset ... \subset F_r=F,
\end{equation}such that the quotients $Q_i=F_i/F_{i-1}$ are stable vector bundles
satisfying $\mu(Q_i)=\mu(F)$ for $1\leq i \leq
r$.}\end{definition}

It can be proved that every semistable vector bundle admits a
Jordan--H\"{o}lder filtration and that all the Jordan--H\"{o}lder
filtrations admitted by a given vector bundle have the same
length. However, there is not a canonical Jordan--H\"{o}lder
filtration associated to a semistable vector bundle $F$. Given a Jordan--H\"{o}lder filtration of a vector bundle, we may
associate to it a canonical object. This is described in the
following definition.

\begin{definition}\textnormal{Consider the direct sum of the stable quotients $\grad
(F)=\oplus_i Q_i$. We call $\grad (F)$ the \emph{graded object}
associated to $F$. This object is canonical in the sense that
$\grad (F)$ is determined up to isomorphism by $F$ (and hence
$Q_1$, ..., $Q_r$ are determined up to order).}\end{definition}

In order to construct the stratification we look at the properties
of the graded object associated to a given vector bundle. The main
object we use is the type, its definition is the following.

\begin{definition}\label{type}\textnormal{We call the $r$-tuple $\underline{n}=( n_1 ,\ldots , n_r)=(
\rank (Q_1) ,\ldots , \rank (Q_r))$ the \emph{type} of the
filtration (\ref{j-h-filtration}). We denote by
$\underline{n}(\sigma)$ the type $( n_{\sigma (1)} ,\ldots ,
n_{\sigma (r)})$, where $\sigma \in S_r$, $S_r$ being the group of
permutation of $r$-elements.}\end{definition}

We will use the type later on in this paper to define a
stratification of the moduli space $G_L(n,d,k)$. Note that the
type is not necessarily determined by $F$.

Consider now the Jordan--H\"{o}lder filtrations
\begin{equation}\label{JH}0=F_0\subset F_1 \subset F_2 \subset \ldots \subset F_r=F,
\end{equation}where $F$ is our usual strictly semistable vector
bundle of rank $n-k$ and degree $d$. We begin by giving
definitions and obtaining results for such a Jordan--H\"older
filtration independently of the length $r$. We essentially provide
conditions for such a filtration to be unique. Unfortunately, we
don't have a description of the sets independently of the $r$.
Later on in this section we restrict ourselves to the case $r$
equals $2$ and $3$. In these cases we obtain complete answers and
descriptions which allow us in the following section to obtain a
stratification for $G_L(n,d,k)$ for the cases in which $n>k$ and
$n-k$ equals 2 and 3.

We have the extensions
\begin{equation}\label{ext-a}0\rightarrow F_i
\rightarrow F_{i+1} \rightarrow F_{i+1}/F_i \rightarrow 0,
\end{equation}
and
\begin{equation}\label{ext-c}0\rightarrow F_i/F_{i-1}
\rightarrow F_{i+1}/F_{i-1} \rightarrow F_{i+1}/F_{i} \rightarrow
0
\end{equation}canonically associated to our Jordan--H\"{o}lder
filtration (\ref{JH}). Here we denote $Q_{i}=F_{i}/F_{i-1}$ and
let $\rank (Q_{i})=n_i$ for all $i$.
\begin{definition}\textnormal{We define $\mathcal{G}_{\underline{n}}$ as
the set of Jordan--H\"older filtrations of type $\underline{n}=(
n_1 ,\ldots , n_r)$, such that the extensions (\ref{ext-a})
associated to the filtration are non-split for every $i$, and
$Q_i \ncong Q_j$ for every $i\neq j$.}\end{definition}

\begin{proposition}\label{sequenceof projective}There is a sequence of projective fibrations for
$\mathcal{G}_{\underline{n}}$, let
$$\mathcal{G}_{\underline{n}}\rightarrow \mathcal{G}_{( n_1 ,\ldots ,
n_{r-1})} \rightarrow \ldots \rightarrow \mathcal{G}_{( n_1
,n_2)}\rightarrow \mathcal{M}_1 \times \ldots \times \mathcal{M}_r
\backslash \Delta_r $$where $\mathcal{M}_i$ is the moduli space of
stable vector bundles of rank $n_i$ and degree $d_i$ and $\Delta_r$
is the ``big diagonal'', that is$$\Delta_r :=\{ (Q_1 , \ldots ,Q_r
)\in \mathcal{M}_1 \times \ldots \times \mathcal{M}_r \textnormal{
$ $ such that $Q_i \cong Q_j$ $ $ for some $ $ $i\neq j$}\}.$$In
particular, when $\gcd(n_i,d_i)=1$ for all $i$, $\mathcal{G}_{\underline{n}}$ parametrizes a universal filtration $$0\subset \mathcal{F}_1 \subset \ldots \subset \mathcal{F}_r.$$\end{proposition}
\begin{proof}We use induction. Let $\widetilde{\mathcal{M}}_i=\widetilde{\mathcal{M}}(n_i,d_i)$ and $\mathcal{M}_i=\mathcal{M}(n_i,d_i)$ be the moduli spaces of (semi)stable vector bundles of rank $n_i$ and degree $d_i$ respectively. The construction depends on the existence of Poincar\'e bundles. In \cite{R} it is proved that when $\gcd (n_i,d_i)\neq 1$ Poincar\'e bundles do not exist over $\mathcal{M}_i$. Then, we need to work at the Quot squeme level. It is well-known that $\widetilde{\mathcal{M}}_i$ can be represented as a GIT quotient of an open subset of a certain Quot scheme, we denote it by $\mathcal{Q}_i$, by the action of an algebraic group (see \cite{N3} Chapter 5 for more details). Let
$R_i^{ss}$ be the open set of $\mathcal{Q}_i$ of semistable points.
If $f_i$ is the morphism from $R_i^{ss}$ to
$\widetilde{\mathcal{M}}_i$, we have that
$(\widetilde{\mathcal{M}}_i,f_i)$ is a good quotient of
$R_i^{ss}$. Let $R^s_i=f_i^{-1}(\mathcal{M}_i)$ and
$f^s_i:R^s_i \rightarrow \mathcal{M}_i$ the restriction of $f_i$.
In this situation, there exist universal bundles
$\mathcal{U}^{ss}_i$ on $R^{ss}_i\times X$. Let $\mathcal{U}^s_i$
be its restriction to $R^s_i \times X$. The group
$GL(N_i)$ acts on $R^s_i$, with the centre acting trivially and
such that $PGL(N_i)$ acts freely. The quotient of $R^s_i$ by
$PGL(N_i)$ is the moduli space of stable bundles $\mathcal{M}_i$.

We do the construction over $$R_1^s\times \ldots \times R_r^s \setminus (f^s_1\times \ldots \times f^s_r)^{-1}\Delta_r.$$The base case is $r=2$. Let $q^s_2:R^s_1\times R^s_2\times X
\rightarrow R^s_1\times R^s_2$ and
$p^s_i:R^s_1\times R^s_2 \rightarrow R^s_i$
for $i=1$, $2$, be the projections. And let $\mathcal{H}^s_2$ be the sheaf
\begin{equation}\label{sheaf_H}\mathcal{R}^1(q^s_2)_{\ast}(\mathscr{H}om((p^s_2\times
id_X)^{\ast}\mathcal{U}^s_2,(p^s_1\times id_X)^{\ast}\mathcal{U}^s_1)),
\end{equation}where $\mathscr{H}om$ is the sheaf of
homomorphisms. Note that Hom$(\mathcal{U}^s_2|_{\{m_2\}\times X},\mathcal{U}^s_1|_{\{m_1\}\times X})=0$ since both are stable bundles of the same slope, then $h^1((\mathcal{U}^s_2|_{\{m_2\}\times X})^{\ast}\otimes (\mathcal{U}^s_1|_{\{m_1\}\times X}))$ is
independent of the choice of the point $(m_1,m_2)\in R^s_1
\times R^s_2\setminus (f^s_1\times f^s_2)^{-1}\Delta_2$. Hence $\mathcal{H}^s_2$ is a bundle on
$R^s_1\times R^s_2\setminus (f^s_1\times f^s_2)^{-1}\Delta_2$.  We consider the projectivization of $\mathcal{H}^s_2$,
$\mathbb{P}(\mathcal{H}^s_2)$. The centre of $GL(N_1)\times GL(N_2)$ acts trivially
on the projective bundle associated to $\mathcal{H}^s_2$ and so
$PGL(N_1)\times PGL(N_2)$ acts freely on
$\mathbb{P}(\mathcal{H}^s_2)$. Using Kempf's descent Lemma (see
\cite{LeP2} page 138, \cite{DN} Theorem 2.3) we obtain that
$\mathbb{P}(\mathcal{H}^s_2)/PGL(N_1)\times PGL(N_2)$ is a projective
fibration over $\mathcal{M}_1 \times \mathcal{M}_2\setminus \Delta_2$ that satisfies
the properties of the proposition. This projective fibration is identified to $\mathcal{G}_{(n_1,n_2)}$. Moreover, let $\mathcal{O}_{P_2}(1)$ be the
tautological bundle of the projective bundle
$\mathbb{P}(\mathcal{H}^s_2)$, $\pi^s_{P_2}:\mathbb{P}(\mathcal{H}^s_2)
\rightarrow \mathcal{U}^s_1\times \mathcal{U}^s_2$ and let
$p_{\mathbb{P}(\mathcal{H}^s_2)}:\mathbb{P}(\mathcal{H}_2^s) \times X
\rightarrow \mathbb{P}(\mathcal{H}_2^s)$ be the projection. We are
in the hypotheses of Remark \ref{remark}, so there exists a
vector bundle $\mathcal{F}^s_2$ over
$\mathbb{P}(\mathcal{H}^s_2)\times X$ and an exact sequence
\begin{equation}\label{ex-dual}0\rightarrow (\pi^s_P \times id_X)^{\ast}(p^s_1 \times
id_X)^{\ast} \mathcal{U}^s_1 \otimes
p_{\mathbb{P}(\mathcal{H}_2^s)}^{\ast}\mathcal{O}_{P_2}(1)\rightarrow
\mathcal{F}^s_2 \rightarrow (\pi^s_P \times id_X)^{\ast}(p^s_2
\times id_X)^{\ast} \mathcal{U}^s_2 \rightarrow 0,
\end{equation}which is universal in the sense of the projective version of Proposition
\ref{Ext-NR}.

In the inductive step we assume that there exists a sequence of projective fibrations $$\mathcal{G}'_{( n_1 ,\ldots ,
n_{r-1})} \rightarrow \ldots \rightarrow \mathcal{G}'_{( n_1
,n_2)}\rightarrow R^s_1 \times \ldots \times R^s_{r-1}
\backslash (f_1^s\times \ldots \times f_{r-1}^s)^{-1}\Delta_{r-1}, $$and a universal family $$0=\mathcal{F}_0^s\subset \mathcal{F}_1^s\subset \ldots \subset \mathcal{F}_{r-1}^s$$parametrized by $\mathcal{G}'_{( n_1 ,\ldots ,
n_{r-1})}$. Note that $\mathcal{F}_1^s$ equals $\mathcal{U}_1^s$. The group $PGL(N_1)\times \ldots \times PGL(N_{r-1})$ acts freely on $\mathcal{G}'_{( n_1 ,\ldots ,
n_{r-1})}$ in such a way that there exists a quotient sequence$$ \mathcal{G}_{( n_1 ,\ldots ,
n_{r-1})} \rightarrow \ldots \rightarrow \mathcal{G}_{( n_1
,n_2)}\rightarrow \mathcal{M}_1 \times \ldots \times \mathcal{M}_{r-1}
\backslash \Delta_{r-1} .$$Let us show that this is also true for $\underline{n}$. By the inductive step we have constructed a sheaf $\mathcal{H}_{r-1}^s$ as
(\ref{sheaf_H}), over $\mathbb{P}(\mathcal{H}_{r-2}^s)\times R^s_{r-1}$. Note
that $\mathcal{H}^s_{r-1}$ is actually a bundle on $\mathbb{P}(\mathcal{H}_{r-2}^s)\times R^s_{r-1}$. We consider the projectivization of $\mathcal{H}^s_{r-1}$,
$\mathbb{P}(\mathcal{H}^s_{r-1})$. One has that $PGL(N_1)\times \ldots \times PGL(N_{r-1})$ acts trivially on $\mathbb{P}(\mathcal{H}^s_{r-1})$, and then it acts trivially on $\mathcal{G}'_{(n_1,\ldots ,n_{r-1})}$, which implies the existence of a sequence of projective fibrations on the quotient. Let $\mathcal{O}_{P_{r-1}}(1)$ be the
tautological bundle of the projective bundle
$\mathbb{P}(\mathcal{H}^s_{r-1})$. Let
$\pi_{P}:\mathbb{P}(\mathcal{H}^s_{r-1}) \rightarrow \mathbb{P}(\mathcal{H}^s_{r-2})\times
R^s_{r-1}$ and let
$p^s_{\mathbb{P}(\mathcal{H}^s_{r-1})}:\mathbb{P}(\mathcal{H}^s_{r-1}) \times X
\rightarrow \mathbb{P}(\mathcal{H}^s_{r-1})$ be the projection. For the existence of $\mathcal{F}_{r-1}^s$ it is required that for every point $(m,m')\in \mathbb{P}(\mathcal{H}^s_{r-2})\times
R^s_{r-1}$ we have that Hom$(\mathcal{U}^s_{r-1}|_{\{m'\}\times X},\mathcal{F}^s_{r-2}|_{\{m\}\times X})=0$ (see Remark \ref{remark}).

Consider now the projections $q^s_r:\mathbb{P}(\mathcal{H}_{r-1}^s)\times
R_r^s\times X \rightarrow \mathbb{P}(\mathcal{H}_{r-1}^s)\times
R_r^s$, $p:\mathbb{P}(\mathcal{H}_{r-1}^s)\times R_r^s
\rightarrow \mathbb{P}(\mathcal{H}_{r-1}^s)$ and
$p^s_r:\mathbb{P}(\mathcal{H}_{r-1}^s)\times R_r^s \rightarrow
R_r^s$. And let $\mathcal{H}_r^s$ be the sheaf
\begin{equation*}\mathcal{R}^1(q^s_r)_{\ast}(\mathscr{H}om((p^s_r\times
id_X)^{\ast}\mathcal{U}^s_r,(p\times id_X)^{\ast}\mathcal{F}^s_{r-1})),
\end{equation*}where $\mathscr{H}om$ is the sheaf of
homomorphisms. Since Hom$(\mathcal{U}^s_r|_{\{ m_3 \} \times X},\mathcal{F}^s_{r-1}|_{\{ m_4 \}
\times X})=0$ for all $(m_4, m_3 )\in
\mathbb{P}(\mathcal{H}^s_{r-1})\times R_r^s$, one has that $\mathcal{H}^s_r$ is a bundle on
$\mathbb{P}(\mathcal{H}^s_{r-1})\times R_r^s$ . Note that if there were a non-zero morphism from $\mathcal{U}^s_r|_{\{ m_3 \} \times X}$ to $\mathcal{F}^s_{r-1}|_{\{ m_4 \}
\times X}$, then one could find a non-zero morphism from $\mathcal{U}_r^s|_{\{ m_3\}\times X}$ to $\mathcal{U}_{r-1}^s|_{\{ m''\}\times X}$ for some $m''\in R_{r-1}^s$, but this is not possible because these are non-isomorphic stable bundles of the same slope. We consider the
projectivization of $\mathcal{H}_r^s$, $\mathbb{P}(\mathcal{H}_r^s)$. One easily see that $PGL(N_1)\times \ldots \times PGL(N_{r})$ acts trivially on $\mathbb{P}(\mathcal{H}^s_{r})$, and then it acts trivially on $\mathcal{G}'_{\underline{n}}$, which implies the existence of the required sequence of projective fibrations on the quotient. Moreover,
let $\mathcal{O}_{P_r}(1)$ be the tautological bundle of the
projective bundle $\mathbb{P}(\mathcal{H}_r^s)$.  Let $\pi_{P_r}:\mathbb{P}(\mathcal{H}_r^s) \rightarrow
\mathbb{P}(\mathcal{H}_{r-1}^s)\times R_r^s$ and let
$p_{\mathbb{P}(\mathcal{H}_r^s)}:\mathbb{P}(\mathcal{H}_r^s) \times X
\rightarrow \mathbb{P}(\mathcal{H}_r^s)$ be the projection.
Then we are in the hypotheses of Remark \ref{remark}, so there
exists a vector bundle $\mathcal{F}_r^s$ over
$\mathbb{P}(\mathcal{H}_r^s)\times X$ and an exact sequence
\begin{equation*}0\rightarrow (\pi_{P_r} \times id_X)^{\ast}(p \times
id_X)^{\ast} \mathcal{F}_{r-1}^s \otimes
p_{\mathbb{P}(\mathcal{H}_r^s)}^{\ast}\mathcal{O}_{P_r}(1)\rightarrow
\mathcal{F}_r^s \rightarrow (\pi_{P_r} \times id_X)^{\ast}(p_r^s \times
id_X)^{\ast} \mathcal{U}_r^s \rightarrow 0,
\end{equation*}which is universal in the sense of the projective version of Proposition
\ref{Ext-NR}.

Finally, when $\gcd (n_i,d_i)=1$ for all $i$ one has that there exists Poincar\'e bundles $\mathcal{P}_i$ over $\mathcal{M}_i$. Then we can repeat the previous argument at the moduli space level obtaining a universal filtration $$0\subset \mathcal{F}_1 \subset \ldots \subset \mathcal{F}_r$$parametrized by $\mathcal{G}_{\underline{n}}$.\end{proof}

\begin{remark}\textnormal{The proof of the previous proposition shows that there always exists a universal filtration at the Quot-scheme level.}\end{remark}

\begin{proposition}\label{JHProposition}The Jordan--H\"older
filtration
\begin{equation}\label{JHHJ}0\subset F_1 \subset F_2 \subset \ldots \subset F_r=F,
\end{equation}of $F$ is unique if and only if no sequence
\begin{equation}\label{split_comment}0\rightarrow Q_i
\rightarrow F_{i+1}/F_{i-1} \rightarrow Q_{i+1} \rightarrow 0
\end{equation}for $0<i\leqslant r-1$, splits. If no two $Q_i$
are isomorphic, this is equivalent to saying that
\begin{equation}Hom (Q_{i+1},F_{i+1}/F_{i-1})=0\end{equation}for $0<i\leqslant
r-1$. \end{proposition}
\begin{proof}For the first statement, suppose that the
Jordan--H\"older filtration (\ref{JHHJ}) is not unique, then if we
have two Jordan--H\"older filtrations for $F$ there exists an
index $i+1<r$ such that $F_j =F'_j$ for all $j<i+1$ and
$F_{i+1}\neq F'_{i+1}$.

If
$F'_{i+1}\subsetneqq F_{i+1}$, there is a non-zero morphism
of vector bundles $\psi :F'_{i+1}/F_i\rightarrow F_{i+1}/F_i$.
Since $F'_{i+1}/F_i$ and $F_{i+1}/F_i$ are stable bundles having
the same slope and $\psi \neq 0$, we get that $F'_{i+1}/F_i \cong
F_{i+1}/F_i$. Hence $F'_{i+1}$ and $F_{i+1}$ have the same rank, so $F'_{i+1}=F_{i+1}$, which is a contradiction. Then $F'_{i+1}\nsubseteq F_{i+1}$. It follows that there exists a unique $j \geqslant i+2$ such that $F'_{i+1}\subset F_j$ but $F'_{i+1}\nsubseteq F_{j-1}$.  This implies that there is a
non-zero bundle morphism $F'_{i+1}\rightarrow F_j/F_{j-1}$, which induces
$\varphi :F'_{i+1}/F_i \rightarrow F_j/F_{j-1}=Q_j$.
Since $\varphi \neq 0$ and $F'_{i+1}/F_i$ and $F_j/F_{j-1}=Q_j$ are
stable bundles having the same slope, then $\varphi$ is an
isomorphism.

The bundle $F_j/F_i$ is
the middle term of the following exact sequence
\begin{equation}\label{split_comment2}0\rightarrow F_{j-1}/F_i
\rightarrow F_{j}/F_{i} \rightarrow Q_{j} \rightarrow 0.
\end{equation}We have that $F'_{i+1}/F_i$ is a subbundle of $F_j/F_i$ which is
isomorphic to $Q_j$, which is stable. One then has that
the sequence (\ref{split_comment2}) splits. It follows that\begin{equation}0\rightarrow Q_{j-1}
\rightarrow F_{j}/F_{j-2} \rightarrow Q_{j} \rightarrow 0
\end{equation}splits.

Suppose now that for some $1\leqslant i\leqslant r$ the sequence
\begin{equation}\label{a}0\rightarrow Q_{i}
\rightarrow F_{i+1}/F_{i-1} \rightarrow Q_{i+1} \rightarrow 0
\end{equation}splits, that is $F_{i+1}/F_{i-1} \cong Q_i \oplus
Q_{i+1}$. We have a Jordan--H\"older filtration of $F$
\begin{equation}\label{JHHJa}0\subset F_1 \subset F_2 \subset \ldots
\subset F_{i-1}\subset F_i \subset F_{i+1}\subset \ldots \subset
F_r=F,
\end{equation}then we can consider the exact sequence
\begin{equation}0\rightarrow F_{i-1} \rightarrow F_i \rightarrow
Q_i \rightarrow 0.
\end{equation}If we take the tensor product by $Q_{i+1}^{\ast}$
and then cohomology, we get the following exact
sequence\begin{equation} H^1(Q_{i+1}^{\ast}\otimes
F_{i-1} )\rightarrow H^1(Q_{i+1}^{\ast}\otimes F_i) \rightarrow
H^1(Q_{i+1}^{\ast}\otimes Q_i) \rightarrow 0.
\end{equation}The fact that (\ref{a}) is split implies that its
extension class in $H^1(Q_{i+1}^{\ast}\otimes Q_i)$ is zero. From
the exactness of the previous sequence, there is an extension \begin{equation}\label{a1}0\rightarrow F_{i-1} \rightarrow F'_i
\rightarrow Q_{i+1} \rightarrow 0
\end{equation}from which the canonical extension \begin{equation*}0\rightarrow F_{i} \rightarrow F_{i+1}
\rightarrow Q_{i+1} \rightarrow 0
\end{equation*}is induced. There is also a commutative diagram
\begin{equation}\xymatrix@R=0.5cm{& 0
 & 0 \\
& Q_i\ar[u]\ar@{=}[r] & Q_i\ar[u] \\
0 \ar[r] &F_i\ar[u]\ar[r]& F_{i+1} \ar[u]\ar[r]& Q_{i+1} \ar[r]& 0 \\
0 \ar[r] &F_{i-1} \ar[u]\ar[r]& F'_{i} \ar[u]\ar[r]&Q_{i+1} \ar@{=}[u]\ar[r]& 0\\
& 0 \ar[u] & 0 \ar[u]}
\end{equation}From this, one gets two different Jordan--H\"older
filtrations of $F$
\begin{equation}\label{JHHJa}0\subset F_1 \subset F_2 \subset \ldots
\subset F_{i-1}\subset F_i \subset F_{i+1}\subset \ldots \subset
F_r=F,
\end{equation}
and
\begin{equation}\label{JHHJb}0\subset F_1 \subset F_2 \subset \ldots
\subset F_{i-1}\subset F'_i \subset F_{i+1}\subset \ldots \subset
F_r=F.
\end{equation}This concludes the proof of the first statement.

For the second statement, if a sequence (\ref{split_comment}) splits and no two $Q_i$
are isomorphic, then the condition Hom$(Q_{i+1},F_{i+1}/F_{i-1})=0$ fails. Conversely, if no two $Q_i$
are isomorphic and there is a non-zero bundle morphism $Q_{i+1}\rightarrow F_{i+1}/F_{i-1}$, then (\ref{split_comment}) splits.
\end{proof}

We introduce now the following subset of
$\mathcal{G}_{\underline{n}}$.
\begin{definition}\label{JHdefinition}\textnormal{We define $\mathcal{E}_{\underline{n}}$ as
the set of bundles which admit Jordan--H\"older filtrations in
$\mathcal{G}_{\underline{n}}$ satisfying
\begin{equation}\label{goodcondition}Hom(Q_{i+1}, F_{i+1}/F_{i-1})=0\end{equation}for every
$i$. Note that $\mathcal{E}_{(n_1,n_2)}\cong
\mathcal{G}_{(n_1,n_2)}$.}\end{definition}

\begin{proposition}$\mathcal{E}_{\underline{n}}$ has a natural
structure of quasi-projective variety.
\end{proposition}
\begin{proof}The conditions (\ref{goodcondition}) are open by the
Semicontinuity Theorem, so this follows from Propositions
\ref{sequenceof projective} and \ref{JHProposition}.\end{proof}

Now, we can calculate the number of parameters on which
$\mathcal{E}_{\underline{n}}$ depends.
\begin{lemma}\label{Parameters}The elements of $\mathcal{E}_{\underline{n}}$ depend
on exactly\begin{equation*}\textnormal{dim}
\widetilde{\mathcal{M}}(n-k,d)-\sum_{1\leq j<i\leq r}n_i n_j (g-1)
\end{equation*}parameters.
\end{lemma}
\begin{proof}We use induction on $r$. The case $r=1$ is trivial.
Assume now that $r\geq 2$ and that the lemma is true for
Jordan--H\"older filtrations of length $r-1$. By Definition
\ref{JHdefinition} and Proposition \ref{JHProposition}, any $F \in
\mathcal{E}_{\underline{n}}$ has a unique Jordan--H\"older
filtration, and in particular there is a non-split extension
\begin{equation*}0\rightarrow F_{r-1} \rightarrow F
\rightarrow Q_r \rightarrow 0
\end{equation*}uniquely determined up to a scalar multiple. Since $Q_i \ncong Q_j$ for
$i\neq j$, we have $h^0 (Q_r^{\ast}\otimes F_{r-1})=0$, so by
Riemann-Roch,$$h^1(Q_r^{\ast}\otimes
F_{r-1})=(n-k-n_r)n_r(g-1).$$By the inductive hypothesis, the
non-split extensions depend on at most$$\textnormal{dim}
\widetilde{\mathcal{M}}(n-k-n_r,d-d_r)-\sum_{1\leq j<i\leq r-1}n_i
n_j (g-1) +n_r^2(g-1)+1+(n-k-n_r)n_r(g-1)-1$$ parameters. It is
easy to check that this coincides with the required formula.
\end{proof}

It will be convenient for our descriptions to use a canonical
filtration associated to our semistable vector bundle that encodes
the information about the Jordan--H\"older filtrations admitted by
this bundle. It turns out that for a given semistable vector
bundle $F$, there is a canonical filtration that satisfies certain
properties as it is proved in the following Lemma.

\begin{lemma}\label{canonical}For every semistable vector
bundle $F$, there is a canonical
filtration\begin{equation}\label{canonical-filtration}
0=E_0\subset E_1\subset E_2 \subset ... \subset E_s=F,
\end{equation}such that the quotients $E_i/E_{i-1}$ are direct sums of stable vector bundles
$E'$ satisfying $\mu(E')=\mu(F)$ for $1\leq i \leq s$ and
$F/E_{i-1}$ contains no subbundle which is the direct sum of
$E_i/E_{i-1}$ with a stable vector bundle of the same slope as
$F$. Actually, if (\ref{j-h-filtration}) is a Jordan--H\"older
filtration of $F$, then $\oplus_{i=1}^{s}E_i/E_{i-1}\cong
\oplus_{j=1}^r F_j/F_{j-1} =\grad F$.
\end{lemma}

\begin{proof}It follows from \cite{K2}, Lemma 3.2.\end{proof}

Later on in this paper, the use of these canonical filtrations
will simplify our descriptions. From now on we restrict our study
to the cases $r=2$ and $r=3$, cases in which we have complete
descriptions.

\subsection{The case $r=2$} Consider the extensions
\begin{equation}\label{F}0\rightarrow Q_1 \rightarrow F \rightarrow
Q_2 \rightarrow 0.
\end{equation}We have $\mu (Q_1)=\mu (F) =\mu (Q_2)$. We denote by $(n_1,d_1)$ and
$(n_2,d_2)$ the invariants of $Q_1$ and $Q_2$ respectively. In
this case $\underline{n}=(n_1,n_2)$ is the type of (\ref{F}) and
$n_1 +n_2 =n-k$. Note that $\grad(F)=Q_1\oplus Q_2$.

\subsubsection{The non-split case} We will classify the non-split
extensions (\ref{F}) in which $F_1$ and $Q_2$ are stable bundles.
As we have already seen, either $Hom(Q_2,Q_1)=0$, or
$Q_1\cong Q_2$. If $Hom(Q_2,Q_1)=0$, then
$h^0(Q_2^{\ast}\otimes Q_1)=0$. If $Q_1\cong Q_2$ then
$h^0(Q_2^{\ast}\otimes Q_1)=1$. Here the quasi-projective variety
$\mathcal{E}_{\underline{n}}$ (see Definition \ref{JHdefinition})
is the space of extension classes of non-splitting extensions
(\ref{F}) satisfying $Hom(Q_2,Q_1)=0$. We need the following
\begin{definition}\textnormal{Let $\mathcal{E}_{\underline{n}}'$ be the space of extension
classes of non-splitting extensions (\ref{F}) satisfying
$Q_1\cong Q_2$.}
\end{definition}
From Proposition \ref{JHProposition} and Definition
\ref{JHdefinition}, we have
\begin{lemma}\label{unicity1}With the above conditions, the non-splitting extension
(\ref{F}) is uniquely determined by $F$ (up to scalar multiples).
In particular the type $\underline{n}$ of (\ref{F}) is determined
by $F$ in this case.
\end{lemma}

We know that the extensions of $Q_2$ by $Q_1$ are classified, up
to equivalence, by $H^1(Q_2^{\ast}\otimes Q_1)$. By Riemann-Roch Theorem
\begin{equation}\label{DIMENSION}h^1(Q_2^{\ast}\otimes
F_1)=n_1(n-k-n_1)(g-1)+h^0(Q_2^{\ast}\otimes F_1).
\end{equation}

We give a complete description of $\mathcal{E}_{\underline{n}}$
and $\mathcal{E}_{\underline{n}}'$.
\begin{proposition}\label{eandep}\begin{itemize}\item[(i)]When $n_1\neq \frac{1}{2}(n-k)$ the
space $\mathcal{E}_{\underline{n}}$ is isomorphic to a projective
bundle over $\mathcal{M}_1\times \mathcal{M}_2$, with fiber the
projective space of dimension $n_1(n-k-n_1)(g-1)-1$ and
$\mathcal{E}_{\underline{n}}'=\emptyset$. \item[(ii)]When
$n_1=\frac{1}{2}(n-k)$:

The space $\mathcal{E}_{\underline{n}}$ is isomorphic to a
projective bundle over $\mathcal{M}_1\times \mathcal{M}_1\setminus
\Delta$, with fiber the projective space of dimension
$n_1^2(g-1)-1$ and where $\Delta := \{ (F',Q)\in \mathcal{M}_1
\times \mathcal{M}_1$ such that $F'\cong Q \}$.

The space $\mathcal{E}_{\underline{n}}'$ is isomorphic to a
projective bundle over $\mathcal{M}_1$, with fiber the projective
space of dimension $n_1^2(g-1)$.
\end{itemize}\end{proposition}

\begin{proof}The construction for $\mathcal{E}_{\underline{n}}$
in both cases appears in the proof of Proposition \ref{sequenceof
projective}.

Regarding $\mathcal{E}_{\underline{n}}'$, when $n_1\neq \frac{1}{2}(n-k)$, one has that $Q_1$ and $Q_2$ are stable bundles of the same slope and different rank, then $Q_1\ncong Q_2$, hence
$\mathcal{E}_{\underline{n}}'=\emptyset$. When $n_1=
\frac{1}{2}(n-k)$, consider first the case in which
$\gcd(n_1,d_1)=1$ and let $\mathcal{M}=\mathcal{M}_1$. Let
$\mathcal{P}$ be the Poincar\'e bundle on $\mathcal{M}\times X$
and $H=\mathscr{H}om((p_2\times id_X)^{\ast}\mathcal{P},(p_1\times
id_X)^{\ast}\mathcal{P})$. Consider the following commutative
diagram
\begin{equation*}\xymatrix{\mathcal{M}'=\mathcal{M}\times X \ar_{q'}[d]
\ar^{\Delta'}[r]&\mathcal{M}\times \mathcal{M}\times X
\ar^{q}[d]\\
\mathcal{M} \ar^(0.5){\Delta}[r]&\mathcal{M}\times \mathcal{M}}
\end{equation*}where $\Delta$ is the diagonal morphism, $q$ and
$q'$ are the natural projections, and $\mathcal{M}'$ is the fiber
product between $\mathcal{M}\times \mathcal{M}\times X$ and
$\mathcal{M}$ over $\mathcal{M}\times \mathcal{M}$.

The pull-back by $\Delta$ of the sheaf
$\mathcal{R}^1q_{\ast} H$ is a locally free sheaf on
$\mathcal{M}$, $\Delta^{\ast}\mathcal{R}^1q_{\ast} H$. Using the
base change formula (\cite[III. \S 9. Proposition 9.3]{H}) we get
that
\begin{equation*}\Delta^{\ast}\mathcal{R}^1q_{\ast} H \simeq \mathcal{R}^1q'_{\ast}\Delta'^{\ast}
H.\end{equation*}So the sheaf
$\mathcal{R}^1q'_{\ast}\Delta'^{\ast} H$ is a bundle on
$\mathcal{M}$ that satisfies all the required properties. By (\ref{DIMENSION}) the projective bundle associated to $\mathcal{R}^1q'_{\ast}\Delta'^{\ast} H$ has dimension $n_1^2(g-1)$.

When the invariants are not coprime an argument similar to the one
we use in the proof of Proposition \ref{sequenceof projective} gives us the result.
\end{proof}

Now, using Lemma \ref{Parameters} we can calculate the number of
parameters on which $\mathcal{E}_{\underline{n}}$ and
$\mathcal{E}_{\underline{n}}'$ depend.
\begin{lemma}The elements of $\mathcal{E}_{\underline{n}}$ depend
on exactly\begin{equation*}\textnormal{dim}
\widetilde{\mathcal{M}}(n-k,d)-n_1(n-k-n_1)(g-1)
\end{equation*}parameters. When $n_1=\frac{1}{2}(n-k)$,
the elements of $\mathcal{E}_{\underline{n}}'$ depend on
exactly\begin{equation*}\textnormal{dim}
\widetilde{\mathcal{M}}(n-k,d)-2n_1^2(g-1)
\end{equation*}parameters.
\end{lemma}

\begin{proof}The first statement is deduced from Lemma \ref{Parameters}.
Note that this computation does not depend on either
$n_1=\frac{1}{2}(n-k)$ or $n_1\neq \frac{1}{2}(n-k)$. For the
numbers of parameters in which $\mathcal{E}_{\underline{n}}'$
depends when $n_1=\frac{1}{2}(n-k)$, the statement follows from an argument
similar to the one used in Lemma \ref{Parameters}.
\end{proof}
\subsubsection{The split case} In this case we consider the bundles $F\cong Q_1 \oplus Q_2 $, such that
 $Q_1$ and $Q_2$ are stable bundles, and $\mu (F)=
\mu (Q_1)=\mu (Q_2)$.
\begin{definition}\textnormal{Let $\mathcal{SE}_{\underline{n}}$ be the space that classifies the bundles $F\cong Q_1\oplus Q_2$ satisfying
$Hom(Q_2,Q_1)=0$. And let $\mathcal{SE}_{\underline{n}}'$ be the
space of those split bundles $F$
satisfying $Q_1\cong Q_2$.}
\end{definition}
\begin{parrafo}\label{parrafo2}\textnormal{When $Q_1 \ncong Q_2 $ and $n_1 \neq \frac{1}{2}(n-k)$, the
bundles $F\cong Q_1\oplus Q_2$ are classified by
$\mathcal{M}_1\times \mathcal{M}_2$. When $n_1 = \frac{1}{2}(n-k)$
and $Q_1 \ncong Q_2 $, then these are classified by
$(\mathcal{M}_1\times \mathcal{M}_1\setminus \Delta )/(\mathbb{Z}/2)$
where the group $\mathbb{Z}/2$ acts permuting the factors.
Finally, when $Q_1 \cong Q_2$, the bundles are
classified by $\mathcal{M}_1$.}
\end{parrafo}
We can again compute the number of parameters on which
$\mathcal{SE}_{\underline{n}}$ and $\mathcal{SE}_{\underline{n}}'$
depend.
\begin{lemma}The elements of $\mathcal{SE}_{\underline{n}}$ depend
on exactly\begin{equation*}\textnormal{dim}
\widetilde{\mathcal{M}}(n_1,d_1)+\textnormal{dim}
\widetilde{\mathcal{M}}(n-k-n_1,d-d_1)
\end{equation*}parameters. When $n_1=\frac{1}{2}(n-k)$,
the elements of $\mathcal{SE}_{\underline{n}}'$ depend on
exactly\begin{equation*}\textnormal{dim}
\widetilde{\mathcal{M}}(n_1,d_1)
\end{equation*}parameters.
\end{lemma}

\subsection{The case $r=3$} When $r=3$, we will classify the different possible sets of
Jordan--H\"{o}lder filtrations that are admitted by our strictly
semistable vector bundles.

When $r=3$, the Jordan--H\"older filtrations admitted by $F$ are
of the form\begin{equation}\label{filr3}0\subset F_1 \subset F_2
\subset F_3 =F.\end{equation}In order to construct
``universal'' filtrations we must construct universal extensions as
we did for $r=2$ (see Proposition \ref{eandep}) in several steps,
which allow us to get universal bundles $F_i$. These bundles could
be split bundles or nonsplit ones.

Let us fix the notation $Q_1=F_1$ and $Q_i=F_i/F_{i-1}$, for all
$i=2, \ldots r$. The bundles $Q_i$ are stable and of the same
slope as $F$.

Let $\underline{n}=(n_1,n_2,n_3)=(\rank (Q_1), \rank (Q_2) ,\rank (Q_3))$ be the type of $F$. We denote by $\underline{n}(\sigma)$ the type
$(n_{\sigma (1)},n_{\sigma (2)},n_{\sigma (3)})$ where $\sigma$ is
a permutation of three elements, for example
$\underline{n}(12)=(n_2,n_1,n_3)$. Assume that the $Q_1$, $Q_2$, $Q_3$  have the same slope. We
assume further that the graded object associated to the semistable vector
bundle $F$ is $\grad F=Q_1\oplus Q_2 \oplus Q_3$. We consider the
following exact sequences
\begin{equation}\label{eq1}
0\rightarrow Q_1\rightarrow F_2\rightarrow Q_2\rightarrow0
\end{equation}
\begin{equation}\label{eq2}
0\rightarrow F_2\rightarrow F\rightarrow Q_3\rightarrow0
\end{equation}and
\begin{equation}\label{eq4}
0\rightarrow Q_2\rightarrow F/Q_1\rightarrow Q_3\rightarrow0,
\end{equation}
canonically associated to the Jordan--H\"older filtration
(\ref{filr3}). Let us denote the classes of these extensions by
$e_1,e_2,\eta $. When $e_i\ne0$, $\eta\ne0$, we write $[e_i]$ and
$[\eta]$ for the corresponding element of the projective space.
Now, the extension classes corresponding to these extensions are
related by the following exact sequence in cohomology
\begin{equation}\label{ext-coho}\ldots \rightarrow \textnormal{Hom}(Q_3,Q_2)
\rightarrow H^1(Q_{3}^{\ast}\otimes Q_{1} )\xymatrix{\ar^{i}[r]&}
H^1(Q_{3}^{\ast}\otimes F_2) \xymatrix{\ar^{p}[r]&}
H^1(Q_{3}^{\ast}\otimes Q_2) \rightarrow 0,
\end{equation}then $\eta =p(e_2)$.

In order to classify the bundles $F$ which arise in this way, we
distinguish the following cases by looking at whether the previous
extensions split or do not. We introduce the
following sets:
\begin{parraf}[\emph{Set 1}]\label{set1}\textnormal{ In this case,
the extensions (\ref{eq1}), (\ref{eq2}) and (\ref{eq4}) are
non-split. From Proposition \ref{JHProposition} the
Jordan--H\"older filtration of $F$ is unique and the bundles $F$
which arise are classified by $5$-tuples
$$Q_1,Q_2,Q_3,[e_1],[e_2].$$Note that in this case, the canonical filtration
(see Lemma \ref{canonical}) coincides with the Jordan--H\"older
filtration.}
\end{parraf}

\begin{parraf}[\emph{Set 2}]\textnormal{ Here, the extensions
(\ref{eq1}) and (\ref{eq2}) are non-split, but (\ref{eq4}) is
split. In this case, the Jordan--H\"older filtration of $F$ is not
unique. There exists an extension \begin{equation}\label{eq3}
0\rightarrow Q_1\rightarrow F_{31}\rightarrow Q_3\rightarrow0,
\end{equation}we denote its extension class in $H^1(Q_3^{\ast}\otimes Q_1)$
by $\eta'$, such that $i(\eta')=e_2$ (see (\ref{ext-coho})).
Then, the bundles $F$ which arise are classified by
$$Q_1,Q_2,Q_3,[e_1],[\eta'],$$
but note that $(Q_1,Q_2,Q_3,[e_1],[\eta'])$ and
$(Q_1,Q_3,Q_2,[\eta'],[e_1])$ give the same $F$. To avoid
duplication, we need to factor out by the action of $\mathbb{Z}/2$
permuting the bundles $Q_2$ and $Q_3$. The canonical filtration in
this case is given by the following exact sequence
\begin{equation}\label{eq40}0\rightarrow Q_1 \rightarrow F
\rightarrow Q_2 \oplus Q_3 \rightarrow 0.\end{equation}From the
canonical filtration we will globalise the construction later on
in this paper.}
\end{parraf}

\begin{parraf}[\emph{Set 3}]\textnormal{ In this case, the extension (\ref{eq1}) is the only one that is
split. The Jordan--H\"older filtration of $F$ is not unique. The
bundles $F$ which arise are classified by
$$Q_1,Q_2,Q_3,[\eta],[\eta'].$$As before, in order to avoid
duplication, we need to factor out by the action of $\mathbb{Z}/2$
permuting the bundles $Q_1$ and $Q_2$. The canonical filtration in
this case is given by the following exact sequence
\begin{equation}\label{1canonicalfilfors3}0\rightarrow Q_1 \oplus Q_2 \rightarrow F
\rightarrow Q_3 \rightarrow 0.\end{equation}}
\end{parraf}

\begin{parraf}[\emph{Set 4}]\textnormal{ The only non-splitting
extension is (\ref{eq1}). Then, the bundle $F$ is
\begin{equation}\label{eq7}
F=F_2\oplus Q_3.
\end{equation}The
bundles $F$ are classified by
$$Q_1,Q_2,Q_3,[e_1].$$The canonical filtration in this case is$$
0\rightarrow Q_1 \oplus Q_3 \rightarrow F_2\oplus Q_3 \rightarrow
Q_2 \rightarrow 0.$$Note that if we interchange $Q_2$ and $Q_3$ we get that $F=F_{31}\oplus Q_2$, which corresponds to the case in which (\ref{eq1}) and (\ref{eq4}) split.}
\end{parraf}

\begin{parraf}[\emph{Set 5}]\textnormal{ Finally, we consider the
case where all the extensions are split. Then
\begin{equation}\label{eq8}
F\cong \grad F =Q_1\oplus Q_2\oplus Q_3.
\end{equation}So the bundles $F$ are classified by $Q_1$, $Q_2$
and $Q_3$. To avoid duplication, we factor out by the action of
$S_3$ permuting the bundles. }
\end{parraf}

If we want to classify the strictly semistable vector bundles in
$\widetilde{\mathcal{M}}(n-k,d)$ of type $\underline
n=(n_1,n_2,n_3)$ that admit a Jordan--H\"older filtration
(\ref{filr3}) and such that $\grad (F)=Q_1\oplus Q_2 \oplus Q_3$
we need also to consider the possibility of $Q_i\cong Q_j$ for
some $i$, $j$. This is accounted for in the following definition:

\begin{definition}\textnormal{Let
\emph{Group 1} be the space whose elements are strictly semistable
vector bundles in $\widetilde{\mathcal{M}}(n-k,d)$ of type
$\underline n=(n_1,n_2,n_3)$ and such that $\grad (F)=Q_1\oplus
Q_2 \oplus Q_3$ where $Q_i\ncong Q_j$ for every $i$, $j$.
Analogously, let \emph{Group 2} be the space in which $Q_i\cong
Q_j$ for two indices $i$ and $j$. Finally, let \emph{Group 3} be
the space in which $Q_i\cong Q_j$ for all $i$ and $j$.}
\end{definition}

As in the case $r=2$ we want to classify in a geometric way all
the possible situations that can appear. In this setup we shall
not have a beautiful description of the spaces of quotients in
terms of projective fibrations. We will still be able to give
some universal constructions in all the cases, but in some of them
only local ones, based always on the results of universal
extensions we introduced in Section \ref{Lange}.

\begin{definition}\label{generalcase}\textnormal{Let
$\mathcal{S}_i^j \mathcal{E}_{\underline n}$ be the space whose
elements are strictly semistable vector bundles in
$\widetilde{\mathcal{M}}(n-k,d)$ of type $\underline
n=(n_1,n_2,n_3)$. The index $j$ means group $j$ and the index $i$
means the set $i$ within the corresponding group. For the elements
of group 2, we need to introduce a couple more indices $\alpha$
and $\beta$. Then $\mathcal{S}_i^2 {\mathcal{E}}_{\underline
n}^{\alpha \beta }$ means that in the graded objects of the
elements of the set, the bundles $Q_{\alpha}$ and $Q_{\beta}$ are
isomorphic. Note that, with the notation of Definition
\ref{JHdefinition} we have that $\mathcal{S}_1^1
\mathcal{E}_{\underline n}\cong \mathcal{E}_{\underline n}$.}
\end{definition}

We are ready now to do our construction. Let
$\mathcal{M}_i=\mathcal{M}(n_i,d_i)$ be the moduli space of stable
bundles of rank $n_i$ and degree $d_i$. Note that
$n_1+n_2+n_3=n-k$ and $d_1+d_2+d_3=d$. We consider here the type
$\underline{n}=(n_1,n_2,n_3)$ and $\underline{n}({\sigma})$ will
be the type obtained from $\underline{n}=(n_1,n_2,n_3)$ after
acting by an element $\sigma \in S_3$. The invariants we have
fixed must satisfy $\frac{d_1}{n_1}=
\frac{d_2}{n_2}=\frac{d_3}{n_3}$.

We are going to construct a ``universal'' Jordan--H\"{o}lder
filtration over $\mathcal{M}_1 \times \mathcal{M}_2 \times
\mathcal{M}_3 \times X$, such that for every point in the base,
i.e. for a fixed graduation, we obtain a Jordan--H\"{o}lder
filtration verifying the required properties. These ``universal''
filtrations will be filtrations associated to the elements of the
different spaces we have defined in Definition \ref{generalcase}.

\subsubsection{The case when $n_1 \neq n_2 \neq n_3$} In this case
it is not possible that $Q_i \cong Q_j$ for some pair
$i\neq j$, so when $n_1,n_2,n_3$ are all distinct,
the spaces $S_i^1{\mathcal E}_{\underline n}$ for $i=1,\ldots ,5$
are the only ones that are non-empty.

The construction for
$\mathcal{S}_1^1\mathcal{E}_{\underline{n}}$ has already been done in the proof of Proposition
\ref{sequenceof projective} when $r=3$. There, the construction is done at the Quot scheme level which implies that this works for any $(n_1,n_2,n_3)$. At the end we use descent lemmas in order to obtain the required construction at the moduli space level.

Here we do all the constructions at the moduli space level assuming the existence of Poincar\'e bundles. This is not true in general. Actually, when $\gcd(n_i,d_i)\neq 1 $ the Poincar\'e bundles do not
exist on $\mathcal{M}_i=\mathcal{M}(n_i,d_i)$. We do the construction at this level for simplicity. When the Poincar\'e bundles do not exist one may do the construction at the Quot scheme level and use descent lemmas afterwards as we did in the proof of Proposition
\ref{sequenceof projective}.

\begin{paragr}{The construction for
$\mathcal{S}_3^1{\mathcal{E}}_{\underline{n}}$.}\label{pa_split}\textnormal{ We have that $\mathcal{M}_i=\mathcal{M}(n_i,d_i)$
for $i=$ $1$, $2$, $3$ and assume that $n_1 < n_2$. Suppose again
that there exist Poincar\'e bundles $\mathcal{P}_1$ and
$\mathcal{P}_2$ on $\mathcal{M}_1\times X$ and
$\mathcal{M}_2\times X$ respectively. Consider also the
projections $p_i:\mathcal{M}_1\times \mathcal{M}_2 \rightarrow
\mathcal{M}_i$ for $i=1$, $2$. There exists a universal vector
bundle $(p_1 \times id_X)^{\ast} \mathcal{P}_1 \oplus (p_2 \times
id_X)^{\ast} \mathcal{P}_2$ over $\mathcal{M}_1\times
\mathcal{M}_2$ in the usual sense. Let $\mathcal{P}_3$ be the
Poincar\'e bundle on $\mathcal{M}_3\times X$. The rest of the
construction is similar to the one in the proof of Proposition
\ref{sequenceof projective}. Let $q':(\mathcal{M}_1\times
\mathcal{M}_2)\times \mathcal{M}_3 \times X \rightarrow
(\mathcal{M}_1\times \mathcal{M}_2)\times \mathcal{M}_3$. Let
$\mathcal{H}'$ be the sheaf
\begin{equation*}\mathcal{R}^1q'_{\ast}(\mathscr{H}om((p_3\times
id_X)^{\ast}\mathcal{P}_3,(p_1 \times id_X)^{\ast} \mathcal{P}_1
\oplus (p_2 \times id_X)^{\ast} \mathcal{P}_2)),
\end{equation*}this is also a bundle on
$(\mathcal{M}_1 \times \mathcal{M}_2 ) \times \mathcal{M}_3$. We
consider the projectivization of $\mathcal{H}'$,
$\mathbb{P}(\mathcal{H}')$. Let $\mathcal{O}_{P'}(1)$ be the
tautological bundle of the projective bundle
$\mathbb{P}(\mathcal{H}')$. For all point $(m_1, m_2, m_3)\in
(\mathcal{M}_1 \times \mathcal{M}_2 ) \times \mathcal{M}_3$, and
for all $m'\in \mathcal{H}'_{(m_1, m_2, m_3)}$, let
$\mathcal{O}_{P'}(1)_{m'} =m'^{\ast}$. Let
$\pi_{P'}:\mathbb{P}(\mathcal{H}') \rightarrow
\mathbb{P}(\mathcal{H})\times \mathcal{M}_3$ and let
$p_{\mathbb{P}(\mathcal{H}')}:\mathbb{P}(\mathcal{H}') \times X
\rightarrow \mathbb{P}(\mathcal{H}')$ be the
projection.}

\textnormal{Now, we want to construct a vector bundle $\mathcal{F}$ over
$\mathbb{P}(\mathcal{H}')\times X$ satisfying all the required
properties. As above, we are in the hypotheses of Remark
\ref{remark}, so there exists a vector bundle $\mathcal{F}$ over
$\mathbb{P}(\mathcal{H}')\times X$ and an exact sequence
\begin{align}\label{canonicalfilfors3}0\rightarrow (\pi_{P'} \times id_X)^{\ast}&\big{(}(p_1 \times id_X)^{\ast}
\mathcal{P}_1 \oplus (p_2 \times id_X)^{\ast} \mathcal{P}_2
\big{)} \otimes
p_{\mathbb{P}(\mathcal{H}')}^{\ast}\mathcal{O}_{P'}(1)\rightarrow
\mathcal{F} \rightarrow \\&\rightarrow (\pi_{P'} \times
id_X)^{\ast}(p_3 \times id_X)^{\ast} \mathcal{P}_3 \rightarrow 0,
\nonumber
\end{align}such that for all $(m_1, m_2, m_3)\in
(\mathcal{M}_1 \times \mathcal{M}_2 ) \times \mathcal{M}_3$, and
for all $m'\in \mathcal{H}'_{(m_1, m_2, m_3)}$, its restriction to
$\{ m' \} \times X$ is the extension
\begin{equation*}0\rightarrow (\mathcal{P}_{1_{m_1}} \oplus \mathcal{P}_{2_{m_2}})\otimes m'^{\ast}
\rightarrow \mathcal{F}_{m'} \rightarrow \mathcal{P}_{3_{m_3}}
\rightarrow 0.
\end{equation*}As a result of this construction we have obtained an extension
(\ref{canonicalfilfors3}) that is the globalising version of the
canonical filtration (\ref{1canonicalfilfors3}). From this
extension we will describe geometrically the corresponding stratum
at the moduli space of coherent systems.}

\textnormal{As in the above case, we must take into account the cases in
which the Poincar\'e bundles do not exist.}

\begin{remark}\textnormal{The construction for $\mathcal{S}_2^1 {\mathcal{E}}_{\underline{n}}$ is obtained by dualising $\mathcal{S}_3^1 {\mathcal{E}}_{\underline{n}}$,
while $\mathcal{S}_4^1 \mathcal{E}_{\underline{n}}$ is simply $\mathcal{E}_{(n_1,n_2)}\times \mathcal{M}_3$. Finally, the construction for $\mathcal{S}_5^1
\mathcal{E}_{\underline{n}}$ is given by $ \mathcal{M}_1\times \mathcal{M}_2\times \mathcal{M}_3$. }
\end{remark}
\end{paragr}

\subsubsection{The case when $n_1 = n_2 \neq n_3$} For the cases in
which the 3-tuple of elements that form the graduations
associated to our semistable vector bundles are elements in
$\mathcal{M}_1 \times \mathcal{M}_1 \times \mathcal{M}_3
\backslash \Delta_{12}$ where $\Delta_{12}$ the diagonal in the
two first components, the constructions we have described for
$\mathcal{S}_i^1\mathcal{E}_{\underline{n}}$ for $i=1,\ldots ,5$
when $n_1 \neq n_2 \neq n_3$ are the same for $n_1 = n_2 \neq
n_3$.

Under the relations between the ranks of the quotient bundles we
also have that $\mathcal{S}_i^3\mathcal{E}_{\underline{n}}$ are
empty.

For the remaining cases, those in which the graduation is an
element of $\Delta_{12} \times \mathcal{M}_3$, a more detailed
study is needed. We describe here the construction of
$\mathcal{S}_1^2\mathcal{E}_{\underline{n}}^{12}$, the rest of the
cases come easily from a suitable combination of the following
construction and the previous ones.

\begin{paragr}{The construction for
$\mathcal{S}_1^2\mathcal{E}_{\underline{n}}^{12}$.} \label{otroparrafo}\textnormal{We want to
construct a sort of universal Jordan--H\"older filtration over
$\Delta_{12} \times \mathcal{M}_3$, where $\Delta_{12}$ is the
diagonal for the two first components, note that in this case
$\mathcal{M}_1 =\mathcal{M}_2$.}

 \textnormal{At the very beginning we restrict ourselves again to a
hypothetical case in which we have Poincar\'e bundles over our
moduli spaces. In spite of the fact that in general this is not
true, we will be able again to work at the Quot scheme level and
afterwards using descent lemmas we will be able to apply our
results at the moduli of stable vector bundles level.}

\textnormal{As in our original
construction (proof of Proposition \ref{sequenceof projective}),
first of all we need to construct a ``universal'' extension over
$\Delta_{12}$. But in this case, a universal extension in the
usual sense (\cite{NR}, \cite{R} \& \cite{S}) does not exist. This
non-existence could be proved bearing in mind that Proposition
\ref{Ext-NR} is a special case of Proposition \ref{Ext-first},
more precisely, the case when $\mathscr{E}xt^0_f
(\mathscr{F},\mathscr{G})=0$ and $\mathscr{E}xt^1_f
(\mathscr{F},\mathscr{G})$ commutes with base change, and the same
for the projective analogues. Here we follow the notation of
Proposition \ref{eandep} (ii), and let $\mathcal{M}_1=\mathcal{M}$
which is a reduced variety. For the morphism $q':\mathcal{M}\times
X \rightarrow \mathcal{M}$, we have that
\begin{align*}\mathscr{E}xt^0_{q'}( \Delta'^{\ast}(p_2 & \times
id_X )^{\ast}\mathcal{P}, \Delta'^{\ast}(p_1\times id_X
)^{\ast}\mathcal{P})\cong \\& \cong R^0 q'_{\ast}\Delta'^{\ast}
\mathscr{H}om ( (p_2\times id_X )^{\ast}\mathcal{P}, (p_1\times
id_X )^{\ast}\mathcal{P}),\end{align*}which is not zero. So there
is not a universal extension in the usual sense.}
\end{paragr}

\begin{theo}\label{Universal}A ``universal'' family of
extensions in the sense of \ref{universal} exists over $\Delta_{12}$.
\end{theo}
\begin{proof}Consider first the following commutative diagram
\begin{equation*}\xymatrix{P\times X \ar_{p'_P}[d]
\ar^{q'_P}[r]&\mathcal{M}\times X
\ar^{q'}[d]\\
P \ar^(0.5){g}[r]&\mathcal{M}}
\end{equation*}where $P=\mathbb{P}(\mathscr{E}xt^1_{q'} ( \Delta'^{\ast}(p_2\times id_X
)^{\ast}\mathcal{P}, \Delta'^{\ast}(p_1\times id_X
)^{\ast}\mathcal{P})^{\ast})$ and $P\times X$ is the fiber product
between $P$ and $\mathcal{M}\times X$ over $\mathcal{M}$.

The existence of this ``universal'' family is based mainly on the
fact that for every $m\in \mathcal{M}$, the base change morphism
\begin{align*}\varphi^1(m):R^1 q'_{\ast}\Delta'^{\ast}&
\mathscr{H}om ( (p_2\times id_X )^{\ast}\mathcal{P}, (p_1\times
id_X )^{\ast}\mathcal{P}) \otimes k(m) \rightarrow \\& \rightarrow
H^1(X_m ,\Delta'^{\ast} \mathscr{H}om ( (p_2\times id_X
)^{\ast}\mathcal{P}, (p_1\times id_X
)^{\ast}\mathcal{P})_m)\end{align*}is surjective. To see this
surjectivity it is enough to note that the fibres of $q'$ are
projective curves. Then, using the Grauert theorem and the
``Cohomology and base change'' theorem (\cite[III. \S 12.
Corollary 12.9 and Theorem 12.11]{H}) we conclude.

Combining the surjectivity of $\varphi^1(m)$ and the ``Cohomology
and Base Change'' theorem we have that $\varphi^i(m)$ are
isomorphisms for $i=0$, $1$. So, because $\mathcal{M}$ is reduced,
we can apply Proposition \ref{ext-projective}. Then, there exists
a family $(e_p)_{p\in P}$ of extensions of
${q'_P}^{\ast}(p_1\times id_X)^{\ast}\mathcal{P}$ by
${q'_P}^{\ast}(p_2\times id_X)^{\ast}\mathcal{P} \otimes
{p'_P}^{\ast}\mathcal{O}_P(1)$ over
$P=\mathbb{P}(\mathscr{E}xt^1_{q'} (\Delta'^{\ast} (p_2\times id_X
)^{\ast}\mathcal{P}, \Delta'^{\ast}(p_1\times id_X
)^{\ast}\mathcal{P})^{\ast})$ which is universal, in the sense of
\ref{universal}, in the category of reduced noetherian
$Y$-schemes for the classes of families of non-splitting
extensions of ${q'_P}^{\ast}(p_1\times id_X)^{\ast}\mathcal{P}$ by
${q'_P}^{\ast}(p_2\times id_X)^{\ast}\mathcal{P} \otimes
{p'_P}^{\ast}\mathscr{L}$ over $S$ with arbitrary $\mathscr{L}\in
Pic(S)$ modulo the canonical operation of
$H^0(S,\mathcal{O}_S^{\ast})$.
\end{proof}

\textnormal{Once we have constructed a family of extensions, $(e_p)_{p\in P}$,
in the first step, for the second we use a similar argument as in
the previous ``universal'' constructions and produce a universal
extension in the usual meaning for each element of the family
$(e_p)_{p\in P}$. Hence, we fix an element of $(e_p)_{p\in P}$,
say
$$0\rightarrow {q'_P}^{\ast}(p_2\times id_X)^{\ast}\mathcal{P}
\otimes {p'_P}^{\ast}\mathcal{O}_P(1)\rightarrow \mathcal{F}_p
\rightarrow {q'_P}^{\ast}(p_1\times
id_X)^{\ast}\mathcal{P}\rightarrow 0.$$Now, as in the usual
notation, let $q'':\{ p \} \times \mathcal{M}_3 \times X
\rightarrow \{ p \} \times \mathcal{M}_3$ and $q_1 : \{ p \}
\times \mathcal{M}_3  \rightarrow \{ p \} $, $p_3 :\{ p \} \times
\mathcal{M}_3 \rightarrow  \mathcal{M}_3$. We have that
\begin{equation*}\mathcal{H}_p =\mathcal{R}^1q''_{\ast}(\mathscr{H}om((p_3\times
id_X)^{\ast}\mathcal{P}_3,(q_1\times id_X)^{\ast}\mathcal{F}_p))
\end{equation*}is a bundle on $\{ p \} \times \mathcal{M}_3$, and
we consider $\mathbb{P}(\mathcal{H}_p)$. To conclude we need to
construct a universal extension on
$\mathbb{P}(\mathcal{H}_p)\times X$. This follows from the fact
that
$$Hom(\mathcal{P}_3|_{\{ m_3 \} \times X},\mathcal{F}_p)=0$$ for all $ m_3 \in
\mathcal{M}_3$, note that in case there exists such a morphism, we
would have another one from $\mathcal{P}_3|_{\{ m_3 \} \times X}$
to $\mathcal{P}|_{\{ m_1 \} \times X}$ for some $m_1 \in
\mathcal{M}$, but this contradicts the hypotheses. Under this
property, the conditions of Proposition \ref{Ext-NR} are fulfilled
(see Remark \ref{remark}) so we have a universal extension in the
usual sense.}

\begin{remark}\textnormal{The constructions for the case when $n_1 = n_2 = n_3$
are analogous to the ones we have described earlier.}\end{remark}

\begin{remark}\textnormal{Regarding the number of parameters on which our sets depend,
from Lemma \ref{Parameters} one obtains that
$\mathcal{S}_1^1\mathcal{E}_{\underline{n}}$ for
$\underline{n}=(n_1,n_2,n_3)$ depends on exactly
$$\dim \widetilde{\mathcal{M}}(n-k,d)
-n_1n_2(g-1)-n_3(n_1+n_2)(g-1).$$Now, to compute the number of
parameters on which the elements of
$\mathcal{S}_3^1{\mathcal{E}}_{\underline{n}}$ depend, it is
enough to look at the extensions of the form$$0\rightarrow
Q_1\oplus Q_2 \rightarrow F \rightarrow Q_3 \rightarrow 0.$$These
extensions depend on exactly $$ \dim
\widetilde{\mathcal{M}}(n-k,d)-(n_1+n_2)n_3(g-1)-2n_1n_2(g-1)+1.$$Finally,
$\mathcal{S}_1^2\mathcal{E}_{\underline{n}}^{12}$ depends on
exactly $$\dim
\widetilde{\mathcal{M}}(n-k,d)-3\dim
\widetilde{\mathcal{M}}(n_1,d_1)-2n_1n_3(g-1)+2.$$The
computations for the remaining cases follow in a similar
fashion.}\end{remark}

\section{A stratification of $G_{L}(n,d,k)$}
\subsection{Defining the stratification}

In this subsection we will define a stratification of the moduli
space $G_L(n,d,k)$ when $k<n$. To this end we use the type that
was defined earlier, and all the sets that we have described from
a geometric point of view in the previous section. The idea is to
define the different strata by looking at the quotient bundle of
the BGN extension associated to every coherent system in $G_L(n,d,k)$.

By Proposition \ref{prop:restate} we know that if the quotient
bundle is stable, the BGN extension gives rise to an
$\alpha$-stable coherent system. If the quotient bundle is only
strictly semistable, the BGN extension could give rise either to an
$\alpha$-stable or a non-$\alpha$-stable coherent system.

In the previous section we studied the sets that classify the
possible Jordan--H\"older filtrations that are admitted by a given
semistable bundle. We define different sets in terms of all the
possible splittings that can appear. These sets will be fundamental to define strata in the
moduli space $G_L(n,d,k)$.

We look at the quotient bundle associated to
our coherent system. The strata are defined
accordingly:
\begin{definition}[The strata]\label{strata}\textnormal{\begin{itemize}\item[(a)]
Using the notation of the previous sections, for the case $r=2$
let $\mathscr{W}_{\mathcal{E}_{\underline{n}}}$ be the space whose
elements are those $(E,V )\in G_L(n,d,k)$ such that if
\begin{equation*}0 \rightarrow \mathcal{O}^{\oplus k}\rightarrow E
\rightarrow F \rightarrow 0
\end{equation*}is the extension that represents the BGN extension
class associated to $(E,V)$ (Proposition \ref{bgmn}), then the
quotient bundle $F$ is strictly semistable, has type
$\underline{n}$ and is an element of
$\mathcal{E}_{\underline{n}}$. We have analogous definitions when
we substitute $\mathcal{E}_{\underline{n}}$ by
$\mathcal{SE}_{\underline{n}}$, $\mathcal{E}_{\underline{n}}'$ and
$\mathcal{SE}_{\underline{n}}'$, respectively. \item[(b)] For the
case $r=3$ we have analogous definitions for the sets we
introduced in Definition \ref{generalcase}. \item[(c)] Let
$\mathscr{W}^1 = G_L(n,d,k) \setminus W$ where $W$ denotes the
subvariety of $G_L(n,d,k)$ consisting of coherent systems for
which the quotient bundle $F$ is strictly semistable.
\end{itemize}}
\end{definition}

\begin{theo}\label{Theo}The sets defined in the previous
definition are locally closed. Moreover, $\mathscr{W}^1$ is an
open set.
\end{theo}

In \cite{BGMMN}, Bradlow \emph{et al.} find a lower bound for the codimension of
$G_L(n,d,k)\setminus \mathscr{W}^1$ in $G_L(n,d,k)$. This is the following:

\begin{lemma}[\cite{BGMMN}, Corollary 7.10]Let $0<k<n$ and suppose that
$G_L(n,d,k)\ne\emptyset$. Then the codimension of
$G_L(n,d,k)\setminus \mathscr{W}^1$ in $G_L(n,d,k)$ is at least
\begin{equation}\label{estimacion}
 \min \{(\sum_{i<j} n_i n_j)(g-1) \},
\end{equation}
where the minimum is taken over all sequences of positive integers
$r,n_1,\ldots ,n_r$ such that $r\ge2$ and $\sum n_i =n-k$.
\end{lemma}

This bound is improved in the following proposition.

\begin{proposition}Let $0<k<n$ and suppose that
$G_L(n,d,k)\ne\emptyset$. When $\gcd (n-k,d)=p\geq 2$ the codimension of
$G_L(n,d,k)\setminus \mathscr{W}^1$ in $G_L(n,d,k)$  is at least
\begin{equation*}
 \frac{p-1}{p^2}(n-k)^2(g-1).
\end{equation*}
\end{proposition}
\begin{proof}In the previous lemma one needs only consider the sequences $n_1,\ldots ,n_r$ for which there exist $d_i$ with $\sum_{i}d_i=d$ such that $\frac{d_i}{n_i}=\frac{d}{n-k}$ for all $i$. This means that each $n_i$ must be a multiple of $\frac{n-k}{p}$. Given this, the minimum of (\ref{estimacion}) is attained when $r=2$ and $n_1=\frac{n-k}{p}$ and $n_2=\frac{(p-1)(n-k)}{p}$ such that $d_1=\frac{d}{p}$ and $d_2=\frac{(p-1)d}{p}$. Then $$ \min \{(\sum_{i<j} n_i n_j)(g-1)\}=\frac{(p-1)}{p^2}(n-k)^2(g-1).$$Hence we conclude.\end{proof}
\subsection{Explicit description of the strata for $r=2$}
\label{determinantal}

In this subsection, we will describe our strata for $r=2$ as complements of
determinantal varieties. As above, the problem is that in general
universal bundles do not exist on our moduli spaces of stable
bundles. Actually, they only exist when the invariants are coprime to
each other. In order to solve this problem, we will work again at
the Quot scheme level -because in these schemes we have universal
families of vector bundles- and afterwards we carry our
construction to the moduli spaces of coherent systems via descent
lemmas. In this case, we assume that the type
is $\underline{n}=(n_1,n-k-n_1)$. We can consider two
different subcases:
\subsubsection{The case when $n_1\neq \frac{1}{2}(n-k)$.}We work again at the Quot scheme level. Using the notations of Section \ref{subsection-epsilon},
let $\mathcal{Q}_i$ be the corresponding Quot schemes, and
$R_i^{s}$ the open set of $\mathcal{Q}_i$ of stable points.
Let $f^s_i$ be the morphism from $R_i^{s}$ to
${\mathcal{M}}_i$.
In this situation, there exist universal bundles
$\mathcal{U}^{s}_i$ on $R^{s}_i\times X$.

In this case we only have two strata, these are $\mathscr{W}_{\mathcal{E}_{\underline{n}}}$ and $\mathscr{W}_{\mathcal{SE}_{\underline{n}}}$. We describe first $\mathscr{W}_{\mathcal{E}_{\underline{n}}}$. As we have done earlier, we are able to construct a
universal extension in the usual sense at the Quot scheme level.
To this end, consider the projections $q^s:R^s_1\times R^s_2\times
X \rightarrow R^s_1\times R^s_2$ and $p^s_i:R^s_1\times R^s_2
\rightarrow R^s_i$ for $i=1$, $2$. Let $\mathcal{H}^s$ be the
sheaf
\begin{equation*}\mathcal{R}^1q^s_{\ast}(\mathscr{H}om((p^s_2\times
id_X)^{\ast}\mathcal{U}^s_2,(p^s_1\times
id_X)^{\ast}\mathcal{U}^s_1)).
\end{equation*}Let $\mathbb{P}(\mathcal{H}^s)$ be the projectivization
of $\mathcal{H}^s$. Let $\pi^s_{P}:\mathbb{P}(\mathcal{H}^s)
\rightarrow \mathcal{U}^s_1\times \mathcal{U}^s_2$ and let
$p_{\mathbb{P}(\mathcal{H}^s)}:\mathbb{P}(\mathcal{H}^s) \times X
\rightarrow \mathbb{P}(\mathcal{H}^s)$ be the projection. We are
again in the hypotheses of Remark \ref{remark}, so there exists a
vector bundle $\mathcal{F}^s$ over
$\mathbb{P}(\mathcal{H}^s)\times X$ and an exact sequence
\begin{equation}\label{ex-dual}0\rightarrow (\pi^s_P \times id_X)^{\ast}(p^s_1 \times
id_X)^{\ast} \mathcal{U}^s_1 \otimes
p_{\mathbb{P}(\mathcal{H}^s)}^{\ast}\mathcal{O}_P(1)\rightarrow
\mathcal{F}^s \rightarrow (\pi^s_P \times id_X)^{\ast}(p^s_2
\times id_X)^{\ast} \mathcal{U}^s_2 \rightarrow 0,
\end{equation}which is universal in the sense of the projective version of Proposition
\ref{Ext-NR}.

Taking the
 dual of (\ref{ex-dual}):
\begin{equation*}0\rightarrow (\pi^s_P \times id_X)^{\ast}(p^s_2 \times
id_X)^{\ast} \mathcal{U}^{s \vee}_2\rightarrow \mathcal{F}^{s
\vee} \rightarrow (\pi^s_P \times id_X)^{\ast}(p^s_1 \times
id_X)^{\ast} \mathcal{U}^{s \vee}_1 \otimes
p_{\mathbb{P}(\mathcal{H}^s)}^{\ast}\mathcal{O}_P(-1)\rightarrow
0,
\end{equation*}and then $\mathcal{R}^i
p_{{\mathbb{P}(\mathcal{H}^s)}{\ast}}$ we have
\begin{align}\label{ext-direct}0\rightarrow \mathcal{R}^1
p_{{\mathbb{P}(\mathcal{H}^s)}{\ast}}&(\pi^s_P \times
id_X)^{\ast}(p^s_2 \times id_X)^{\ast} \mathcal{U}^{s \vee}_2
\rightarrow \mathcal{R}^1
p_{{\mathbb{P}(\mathcal{H}^s)}{\ast}}\mathcal{F}^{s\vee}
\rightarrow
\\& \rightarrow
\mathcal{R}^1 p_{{\mathbb{P}(\mathcal{H}^s)}{\ast}}(\pi^s_P \times
id_X)^{\ast}(p^s_1 \times id_X)^{\ast} \mathcal{U}^{s\vee}_1
\otimes \mathcal{O}_P(-1)\rightarrow 0.\nonumber
\end{align}To simplify this extension, we introduce the following diagram
\begin{equation*}\xymatrix{\mathbb{P}(\mathcal{H}^s)\times
X\ar[d]_{p_{\mathbb{P}(\mathcal{H}^s)}}\ar[rr]^{\pi^s_{P}\times
id_X} &&\mathcal{R}^s_1\times \mathcal{R}^s_2\times X \ar[d]^{q^s}
\ar[rr]^{p^s_i\times id_X}&&\mathcal{R}^s_i \times X\ar[d]^{\pi^s_i}\\
\mathbb{P}(\mathcal{H}^s)\ar[rr]^{\pi^s_{P}}
&&\mathcal{R}^s_1\times \mathcal{R}^s_2
\ar[rr]^{p^s_i}&&\mathcal{R}^s_i }
\end{equation*}Using again the base change formula we have
$$\mathcal{R}^1 p_{{\mathbb{P}(\mathcal{H}^s)}{\ast}}(\pi^s_P
\times id_X)^{\ast}(p^s_i \times id_X)^{\ast} \mathcal{U}^{s
\vee}_i\cong \pi^{s\ast}_P p^{s\ast}_i\mathcal{R}^1\pi^s_{i
\ast}\mathcal{U}^{s\vee}_i$$so the extension (\ref{ext-direct}) is
\begin{equation}\label{ext-direct-final}0\rightarrow \pi_P^{s\ast}p_2^{s\ast}\mathcal{R}^1
\pi^s_{2\ast}\mathcal{U}_2^{s\vee}\rightarrow \mathcal{R}^1
p_{{\mathbb{P}(\mathcal{H}^s)}{\ast}}\mathcal{F}^{s\vee}
\rightarrow
\pi_P^{\ast}p_1^{s\ast}\mathcal{R}^1\pi^s_{1\ast}\mathcal{U}_1^{s\vee}
\otimes \mathcal{O}_P(-1)\rightarrow 0.
\end{equation}
Consider now the set $W_{\mathcal{E}_{\underline{n}}}:= \{(e_1,
e_2, e)$ where $(e_1 ,e_2 )\in \mathcal{R}^s_1\times
\mathcal{R}^s_2$ and $e\in \mathbb{P}(\mathcal{H}^s_{(e_1
,e_2)})\}$. Consider the Grassmann bundle of $k$-planes of the
bundle $\mathcal{R}^1
p_{{\mathbb{P}(\mathcal{H}^s)}{\ast}}\mathcal{F}^{s\vee}$, let
$\Gr(k,\mathcal{R}^1
p_{{\mathbb{P}(\mathcal{H}^s)}{\ast}}\mathcal{F}^{s\vee})$. For
every point $w\in W_{\mathcal{E}_{\underline{n}}}$ we define the
following determinantal variety
\begin{equation*}V_w:=\big{\{} \pi \in \Gr(k,\mathcal{R}^1
p_{{\mathbb{P}(\mathcal{H}^s)}{\ast}}\mathcal{F}^{s\vee})_w :
\dim\big{(}\pi \cap (\pi_P^{s\ast}p_2^{s\ast}\mathcal{R}^1
\pi^s_{2\ast}\mathcal{U}^{s\vee}_2)_w\big{)}\geqslant k
(1-\frac{n_1}{n-k}) \big{\}}.
\end{equation*}

Let $V_{\mathcal{E}_{\underline{n}}}:=\coprod_{w\in
W_{\mathcal{E}_{\underline{n}}}} V_w \subseteq \coprod_{w\in
W_{\mathcal{E}_{\underline{n}}}} \Gr(k,\mathcal{R}^1
p_{{\mathbb{P}(\mathcal{H}^s)}{\ast}}\mathcal{F}^{s\vee})_w$, this
is a family of determinantal varieties.

Now, from the proofs of Proposition \ref{sequenceof projective} and Proposition
\ref{eandep} (i), we have that $\mathbb{P}(\mathcal{H}^s)/PGL(N_1)\times
PGL(N_2)$ is a projective fibration over $\mathcal{M}_1 \times
\mathcal{M}_2$. Because the scheme
$V_{\mathcal{E}_{\underline{n}}}$ is closed and invariant under
the action of $PGL(N_1)\times PGL(N_2)$, using Kempf's descent
Lemma, $V_{\mathcal{E}_{\underline{n}}}$ descends to a projective scheme over $\mathbb{P}(\mathcal{H}^s)/PGL(N_1)\times PGL(N_2)$,
which we call $\mathcal{V}_{\mathcal{E}_{\underline{n}}}$. If we
denote by $\mathcal{V}_{\mathcal{E}_{\underline{n}}}^c$ the
complement of $\mathcal{V}_{\mathcal{E}_{\underline{n}}}$ in
$$\coprod_{w\in W_{\mathcal{E}_{\underline{n}}}}
\Gr(k,\mathcal{R}^1
p_{{\mathbb{P}(\mathcal{H}^s)}{\ast}}\mathcal{F}^{s\vee})_w /
PGL(N_1)\times PGL(N_2),$$we have the following
\begin{theo}\label{deter}The stratum $\mathscr{W}_{\mathcal{E}_{\underline{n}}}$ is
identified with $\mathcal{V}_{\mathcal{E}_{\underline{n}}}^c$.
\end{theo}
\begin{proof}This follows from the previous construction and Theorem \ref{bgn}.
\end{proof}

\begin{remark}\label{nota}\textnormal{ Note that there is an action of the group of automorphisms of $\mathcal{F}^{s}_w$, we denote it by $\mathcal{G}_{\mathcal{E}_{\underline{n}}}$, on $(\mathcal{R}^1
p_{{\mathbb{P}(\mathcal{H}^s)}{\ast}}\mathcal{F}^{s\vee})_w$, that is reflected in $\mathcal{V}_{\mathcal{E}_{\underline{n}}}^c$. The stratum $\mathscr{W}_{\mathcal{E}_{\underline{n}}}$ is then identified to the quotient of $\mathcal{V}_{\mathcal{E}_{\underline{n}}}^c$ by $\mathcal{G}_{\mathcal{E}_{\underline{n}}}$, but in this case this quotient is equal to $\mathcal{V}_{\mathcal{E}_{\underline{n}}}^c$ since $\mathcal{G}_{\mathcal{E}_{\underline{n}}}$ equals $\mathbb{C}^{*}$. For the rest of the strata the corresponding groups of automorphisms are not trivial and need to be taken into account.}\end{remark}

Regarding the stratum $\mathscr{W}_{\mathcal{SE}_{\underline{n}}}$, one may consider the universal bundle $(p_1^s\times id_X)^{\ast}\mathcal{U}_1^s\oplus
(p_2^s\times id_X)^{\ast}\mathcal{U}_2^s $ over $R_1^s\times R_2^s \times X$.
Consider the Grassmann bundle of $k$-planes of the
bundle $\mathcal{R}^1q^s_{\ast}(
(p_1^s\times id_X)^{\ast}\mathcal{U}_1^{s\vee}\oplus
(p_2^s\times id_X)^{\ast}\mathcal{U}_2^{s\vee})$, let
$\Gr^{\mathcal{SE}_{\underline{n}}}:=\Gr(k,\mathcal{R}^1q^s_{\ast}(
(p_1^s\times id_X)^{\ast}\mathcal{U}_1^{s\vee}\oplus
(p_2^s\times id_X)^{\ast}\mathcal{U}_2^{s\vee}))$. For
every point $w=(r_1,r_2)\in R_1^s\times R_2^s$ we define the
determinantal varieties
\begin{equation*}V^1_w:=\big{\{} \pi \in \Gr^{\mathcal{SE}_{\underline{n}}}_w :
\dim\big{(}\pi \cap H^1(\mathcal{U}^{s\vee}_2|_{\{r_2\} \times X})\big{)}\geqslant k
(1-\frac{n_1}{n-k}) \big{\}},
\end{equation*}and\begin{equation*}V^2_w:=\big{\{} \pi \in \Gr^{\mathcal{SE}_{\underline{n}}}_w :
\dim\big{(}\pi \cap H^1(\mathcal{U}^{s\vee}_1|_{\{r_1\} \times X})\big{)}\geqslant k
(1-\frac{n-k-n_1}{n-k}) \big{\}}.
\end{equation*}
Let $V_{\mathcal{SE}_{\underline{n}}}:=\coprod_{w\in
R_1^s\times R_2^s} (V^1_w\cup V^2_w) \subseteq \coprod_{w\in
R_1^s\times R_2^s} \Gr^{\mathcal{SE}_{\underline{n}}}_w $, this
is again a family of determinantal varieties. Using a descent
argument, $V_{\mathcal{SE}_{\underline{n}}}$ descends to a closed scheme over $\mathcal{M}_1\times \mathcal{M}_2$,
which we call $\mathcal{V}_{\mathcal{SE}_{\underline{n}}}$. If we
denote by $\mathcal{V}_{\mathcal{SE}_{\underline{n}}}^c$ the
complement of $\mathcal{V}_{\mathcal{SE}_{\underline{n}}}$ in
$\coprod_{w\in R_1^s\times R_2^s}
\Gr^{\mathcal{SE}_{\underline{n}}}_w /
PGL(N_1)\times PGL(N_2)$. The group of automorphisms of $\big{(}(p_1^s\times id_X)^{\ast}\mathcal{U}_1^{s\vee}\oplus
(p_2^s\times id_X)^{\ast}\mathcal{U}_2^{s\vee}\big{)}_w$, we denote it by $\mathcal{G}_{\mathcal{SE}_{\underline{n}}}$, acts on $\big{(}\mathcal{R}^1q^s_{\ast}(
(p_1^s\times id_X)^{\ast}\mathcal{U}_1^{s\vee}\oplus
(p_2^s\times id_X)^{\ast}\mathcal{U}_2^{s\vee}))\big{)}_w$ and induces an action on $\mathcal{V}_{\mathcal{SE}_{\underline{n}}}^c$. Then
\begin{theo}\label{deter1}The stratum $\mathscr{W}_{\mathcal{SE}_{\underline{n}}}$ is
identified with  $\mathcal{V}_{\mathcal{SE}_{\underline{n}}}^c/\mathcal{G}_{\mathcal{SE}_{\underline{n}}}$.
\end{theo}


\subsubsection{The case when $n_1 =\frac{1}{2}(n-k)$.} For
$\mathcal{E}_{\underline{n}}$ and $\mathcal{SE}_{\underline{n}}$ the construction is the same as
before. For $\mathcal{E}'_{\underline{n}}$ and $\mathcal{SE}'_{\underline{n}}$, we only do the
construction in the case in which the invariants are coprime, the
remaining cases follow easily from the forthcoming construction
and the previous one.

We do first
$\mathscr{W}_{\mathcal{E}'_{\underline{n}}}$. We need a ``universal'' extension, but in
this case, as we saw in Paragraph \ref{otroparrafo}, it does not
exist. To solve this problem, in Theorem \ref{Universal} we proved
the existence of a family $(e_p)_{p\in P}$ of extensions of
${q'_P}^{\ast}(p_1\times id_X)^{\ast}\mathcal{P}$ by
${q'_P}^{\ast}(p_2\times id_X)^{\ast}\mathcal{P} \otimes
{p'_P}^{\ast}\mathcal{O}_P(1)$ over
$P=\mathbb{P}(\mathscr{E}xt^1_{q'} (\Delta'^{\ast} (p_2\times id_X
)^{\ast}\mathcal{P}, \Delta'^{\ast}(p_1\times id_X
)^{\ast}\mathcal{P})^{\ast})$ which is universal in the sense of Subsection
\ref{universal}.

This means that we have a local universal family of
extensions, instead of the universal extension that we were
allowed to construct in the case of the stratum induced by
$\mathcal{E}_{\underline{n}}$. We denote $\mathcal{M_1}=\mathcal{M}$. So as we did in the case of
$\mathcal{E}_{\underline{n}}$, we consider the set $W:=\{
(m,h):m\in \mathcal{M}$ and $h\in \mathbb{P}(\mathcal{R}^1
q'_{\ast}\Delta'^{\ast} \mathscr{H}om ( (p_2\times id_X
)^{\ast}\mathcal{P}, (p_1\times id_X )^{\ast}\mathcal{P})_m) \}$.
For every point $w=(m,h)\in W$ we have an extension
\begin{equation*}0\rightarrow F' \otimes \mathcal{O}_P(1)_m\rightarrow F \rightarrow
F'\rightarrow 0,
\end{equation*}taking the dual and cohomology, we get
\begin{equation*}0\rightarrow H^1(F'^{\vee})\rightarrow H^1(F^{\vee}) \rightarrow
H^1(F'^{\vee}) \otimes \mathcal{O}_P(-1)_m\rightarrow 0.
\end{equation*}

We define the following variety
\begin{equation*}V_w:=\big{\{} \pi \in \Gr(k,H^1(F^{\vee})) :
dim\big{(}\pi \cap H^1(F'^{\vee})\big{)} \geqslant \frac{k}{2}
\big{\}},
\end{equation*}and let $\mathcal{V}_{\mathcal{E}'_{\underline{n}}}:=\coprod_{w\in
W} V_w \subseteq \coprod_{w\in W} \Gr(k,H^1(F^{\vee}))$. Again, the group of automorphisms of $F$, $\mathcal{G}_{\mathcal{E}'_{\underline{n}}}=\Aut F$, acts on $H^1(F^{\vee})$ and from its induced action on $\mathcal{V}_{\mathcal{E}'_{\underline{n}}}^c$ we get
\begin{theo}\label{deter3}The stratum $\mathscr{W}_{\mathcal{E}'_{\underline{n}}}$ is
identified with $\mathcal{V}_{\mathcal{E}'_{\underline{n}}}^c/\mathcal{G}_{\mathcal{E}'_{\underline{n}}}$.
\end{theo}

Regarding the stratum
$\mathscr{W}_{\mathcal{SE}'_{\underline{n}}}$, we have the universal bundle $(p_1\times id_X)^{\ast}\mathcal{P}\oplus
(p_2\times id_X)^{\ast}\mathcal{P}$ over $\mathcal{M} \times X$. For every $m\in \mathcal{M}$, each $f\in \mathbb{P}^1$ defines a morphism of vector bundles
\begin{equation*}Q\xymatrix{\ar^f[r]&} Q \oplus Q,
\end{equation*}where $Q=\mathcal{P}|_{\{m\}\times X}$. Taking the dual and $H^i$ we get
$H^1(f^{\vee}): H^1(Q^{\vee}) \oplus H^1(Q^{\vee}) \rightarrow
H^1(Q^{\vee})$. We then define the following variety
\begin{equation*}V^f_m:=\big{\{} \pi \in \Gr(k,H^1(Q^{\vee}) \oplus H^1(Q^{\vee}) ) :
\dim \big{(}\pi \cap \ker H^1(f^{\vee})\big{)} \geqslant \frac{k}{2}
\big{\}},
\end{equation*}and let $\mathcal{V}_{\mathcal{SE}'_{\underline{n}}}:=\coprod_{(f,m)\in
\mathbb{P}^1\times \mathcal{M}} V^f_m \subseteq \coprod_{(f,m)\in
\mathbb{P}^1\times \mathcal{M}} \Gr(k,H^1(Q^{\vee})\oplus H^1(Q^{\vee}))$. The group of automorphisms of $Q\oplus Q$, $\mathcal{G}_{\mathcal{SE}'_{\underline{n}}}=\Aut (Q\oplus Q)$, acts on $H^1(Q^{\vee})\oplus H^1(Q^{\vee})$ and from its induced action on $\mathcal{V}_{\mathcal{SE}'_{\underline{n}}}^c$ one obtains
\begin{theo}\label{deter4}The stratum $\mathscr{W}_{\mathcal{SE}'_{\underline{n}}}$ is
identified with $\mathcal{V}_{\mathcal{SE}'_{\underline{n}}}^c/\mathcal{G}_{\mathcal{SE}'_{\underline{n}}}$.
\end{theo}

\begin{remark}\label{fibrationStrata}\textnormal{Using the construction above
we see that one may describe the varieties corresponding to our
strata at the Quot scheme level as locally trivial fiber bundles
(in the Zariski topology). For instance, if we look at the stratum
$\mathscr{W}_{\mathcal{E}_{\underline{n}}}$, this is isomorphic to
the descended variety corresponding to
$W_{\mathcal{E}_{\underline{n}}}$ by the morphism
$$f_1^s \times f_2^s :\mathcal{R}_1^s \times \mathcal{R}_1^s
\rightarrow \mathcal{M}_1 \times \mathcal{M}_2.$$The variety
$W_{\mathcal{E}_{\underline{n}}}$ is isomorphic to a locally
trivial fibration (in the Zariski topology) over the projective
fibration $\mathbb{P}(\mathcal{H}^s)$, over $\mathcal{R}^s_1
\times \mathcal{R}^s_2$ that appears in the proof of Proposition
\ref{sequenceof projective}. The fiber of our locally trivial
fiber bundle is ${V}^c={V}_w^c$, that is the complement of $V=V_w$
in $\Gr(k,\mathcal{R}^1
p_{{\mathbb{P}(\mathcal{H}^s)}{\ast}}\mathcal{F}^{s\vee})_w \cong
\Gr (k, d+(n-k)(g-1))$. Moreover, both fibrations are invariant
for the action of $PGL(N_1)\times PGL(N_2)$. For the rest of the
strata one gets the same sort of description. }
\end{remark}

\subsection{Irreducibility of the strata}

In this subsection we will prove that the strata we have defined
earlier in this paper are irreducible.

\begin{theo}\label{irre_smooth}The strata described in Definition \ref{strata} are irreducible.
\end{theo}
\begin{proof}The irreducibility condition for $\mathscr{W}^1$ comes directly from
Proposition \ref{prop:restate}. The argument we use to prove that
the rest of our strata are irreducible is the same for every
stratum so we prove it for the simplest case. In Proposition
\ref{eandep} (i) we proved that when $n_1\neq \frac{1}{2}(n-k)$
the space $\mathcal{E}_{\underline{n}}$ is isomorphic to a
projective bundle over $\mathcal{M}_1\times \mathcal{M}_2$ of
constant dimension. Now, because $\mathcal{M}_1\times
\mathcal{M}_2$ is irreducible, we have that
$\mathcal{E}_{\underline{n}}$ is also irreducible. By Theorem
\ref{deter} we have defined an open family of extensions within
$\mathcal{E}_{\underline{n}}$, hence this family is again
irreducible and maps into $G_L(n,d,k)$. Its image is irreducible
and is identified with
$\mathscr{W}_{\mathcal{E}_{\underline{n}}}$.
\end{proof}

\section{Hodge--Poincar\'e polynomials}

We use Deligne's extension of Hodge theory which applies to
varieties which are not necessarily compact, projective or smooth
(see \cite{D2}, \cite{D3} and \cite{D4}). We start by giving a
review of the notions of pure Hodge structure, mixed Hodge
structure, Hodge--Deligne and Hodge--Poincar\'e polynomials
under these general hypotheses.

\begin{definition}\label{pure}\textnormal{A \emph{pure Hodge structure of weight $m$} is given by
a finite dimensional $\mathbb{Q}$--vector space $H_{\mathbb{Q}}$
and a finite decreasing filtration $F^p$ of
$H=H_{\mathbb{Q}}\otimes \mathbb{C}$
$$H\supset \ldots \supset F^p \supset \ldots \supset ( 0 ) , $$
called \emph{the Hodge filtration}, such that $H=F^p\oplus
\overline{F^{m-p+1}}$ for all $p$. When $p+q=m$, if we set
$H^{p,q}=F^p \cap \overline{F^q}$, the condition $H=F^p\oplus
\overline{F^{m-p+1}}$ for all $p$ implies an equivalent definition
for a pure Hodge structure. That is, a decomposition
$$H=\bigoplus_{p+q=m}H^{p,q}$$
satisfying that $H^{p,q}=\overline{H^{q,p}}$, where
$\overline{H^{q,p}}$ is the complex conjugate of $H^{q,p}$. The
relation between the two equivalent definitions is the following:
Given a filtration $\{ F^p\}_p$ we obtain a decomposition by
considering $H^{p,q}=F^p\cap \overline{F^q}$. Given a
decomposition $\{ H^{p,q} \}_{p,q}$, this defines a filtration as
above by $F^{p}=\bigoplus_{i\geq p} H^{i,m-i}.$}
\end{definition}

The $n$-th cohomology group of a smooth projective variety $H^n(X)$
carries a pure Hodge structure of weight $n$. If
$\Omega_X^{\bullet}$ denote the complex of holomorphic
differential forms, and $(\Omega_X^{\bullet})^{\geq p}$ is the subcomplex
of forms of degree greater than or equal to $p$. Let $\mathbb{H}(X,\Omega_X^{\bullet})$ be the hypercohomology of the complex $\Omega_X^{\bullet}$, then one has that $H^n(X,\mathbb{C})=\mathbb{H}(X,\Omega_X^{\bullet})$.
The role of the
Hodge filtration is played here by the following filtration:
$$F^p=\im (\mathbb{H}^n(X,(\Omega_X^{\bullet})^{\geq p}) \rightarrow
\mathbb{H}^n(X,\Omega_X^{\bullet})).$$

\begin{parrafo}\textnormal{A morphism of Hodge structures is a map
$f_{\mathbb{Q}}:H_{\mathbb{Q}} \rightarrow H_{\mathbb{Q}}'$ such
that $f_{\mathbb{C}}(F^p H)\subset F^p H'$ for all $p$, where
$f_{\mathbb{C}}=f_{\mathbb{Q}}\otimes \mathbb{C}$ and $F^p H$ is
the $p$-th element in the Hodge filtration of $H$. When the Hodge
structures have the same weight, $f_{\mathbb{Q}}$ strictly
preserves the filtration, that is
$$\im (f_{\mathbb{C}})\cap F^p H' =f_{\mathbb{C}}(F^p H).$$
It is also known that for a given weight, the pure Hodge
structures form an abelian category.}\end{parrafo}

\begin{definition}\label{MHS}\textnormal{A \emph{mixed Hodge structure} consists
of a finite dimensional $\mathbb{Q}$-vector space
$H_{\mathbb{Q}}$, an increasing filtration $W_l$ of
$H_{\mathbb{Q}}$, called \emph{the weight filtration}$$\ldots
\subset W_l \subset \ldots \subset H_{\mathbb{Q}},$$and the Hodge
filtration $F^p$ of $H=H_{\mathbb{Q}}\otimes \mathbb{C}$, where
the filtrations $F^p Gr_l^W$ induced by $F^p$ on $$Gr_l^W=(W_l
H_{\mathbb{Q}} / W_{l-1}H_{\mathbb{Q}})\otimes \mathbb{C} = W_l H
/ W_{l-1}H$$ give a pure Hodge structure of weight $l$. Here $F^p
Gr_l^W$ is given by $$(W_lH\cap F^p +W_{l-1}H)/W_{l-1}H.$$}
\end{definition}

\begin{parrafo}\textnormal{A morphism of type $(r,r)$ between
mixed Hodge structures, $H_{\mathbb{Q}}$ with filtrations $W_m$
and $F^p$, and $H_{\mathbb{Q}}'$ with $W_l'$ and ${F'}^q$, is
given by a linear map $$L:H_{\mathbb{Q}} \rightarrow
H_{\mathbb{Q}}'$$ satisfying $L(W_m)\subset W'_{m+2r}$ and
$L(F^p)\subset F'^{p+r}$. Any such morphism is then strict in the
sense that $L(F^p)=F'^{p+r}\cap \im(L)$, and the same for the
weight filtration.}\end{parrafo}

\begin{definition}\textnormal{A morphism of type $(0,0)$ between
mixed Hodge structures, is called \emph{a morphism of mixed Hodge
structures}.}\end{definition}

Our main interests in this paper are the cohomology groups
$H^k(X,\mathbb{Q})$ of a complex variety $X$ which may be singular
and not projective. Deligne proved that these groups carry a mixed
Hodge structure (see \cite{D2}, \cite{D3} and \cite{D4}).
Associated to the Hodge filtration and the weight filtration we
can consider the quotients $Gr_l^W=W_l/W_{l-1}$ of Definition
\ref{MHS}, and for the Hodge filtration
$Gr_F^pGr_l^W=F^pGr_l^W/F^{p+1}Gr_l^W$. Deligne also proved that
the cohomology groups with compact support, we denote them by
$H^k_c(X)$, carry a mixed Hodge structure (see \cite{D2},
\cite{D3} and \cite{D4}). We can then define the Hodge--Deligne
numbers of $X$ as follows

\begin{definition}\textnormal{For a complex algebraic variety $X$, not
necessarily smooth, compact or irreducible, we define its
\emph{Hodge--Deligne numbers} as $$h^{p,q}(H_c^k(X))=\dim Gr_F^p
Gr_{p+q}^W H_c^k(X).$$}
\end{definition}

We may introduce the following Euler characteristic
\begin{equation}\label{ECC}\chi_{p,q}^c(X)=\sum_k(-1)^k
h^{p,q}(H_c^k(X)).\end{equation}We write $\chi_{p,q}(X)$ for the
Euler characteristic (\ref{ECC}) of $H^k(X)$. Then under the
hypothesis of $X$ being smooth of dimension $n$, Poincar\'{e}
duality tells us that $$\chi_{p,q}^c(X)=\chi_{n-p,n-q}(X).$$ We
are now ready to define the Hodge--Deligne polynomial.

\begin{definition}[\cite{DK}]\textnormal{For any complex algebraic variety $X$,
we define its \emph{Hodge--Deligne polynomial} (or virtual Hodge
polynomial) as
$$\mathcal{H}(X)(u,v)=\sum_{p,q}(-1)^{p+q}\chi_{p,q}^c(X)u^p v^q \in \mathbb{Z}[u,v].$$}
\end{definition}

Danilov and Khovanski\v{i} (\cite{DK}) observed that
$\mathcal{H}(X)(u,v)$ coincides with the classical Hodge
polynomial when $X$ is smooth and projective. Note that under
these hypotheses, the mixed Hodge structure on $H_c^k(X)$ is pure
of weight $k$, so
\begin{equation*}
Gr_m^W H^k_c(X)=\left \{ \begin{array}{l@{}l} H^k(X) \textnormal{$
$ $ $ if $ $ $m=k$.}
\\ 0 \textnormal{$ $ $ $ if $ $ $m\neq k$.}
\end{array} \right .
\end{equation*}
Then
\begin{equation}\label{UHP}\mathcal{H}(X)(u,v)=\sum_{p,q}h^{p,q}(X)u^p v^q,
\end{equation}where $h^{p,q}(X)=h^{p,q}(H^{p+q}(X))$ are the classical Hodge numbers
of $X$ and (\ref{UHP}) the classical Hodge polynomial.

We may define another polynomial using the Euler characteristic
$\chi_{p,q}(X)$ for rational cohomology groups without compact
support. As we have already said Deligne proved that these groups
carry a mixed Hodge structure with the usual given associated
filtrations.

\begin{definition}\textnormal{For a complex algebraic variety $X$, not
necessarily smooth, compact or irreducible, we define its
\emph{Hodge--Poincar\'{e} numbers} as $$h^{p,q}(H^k(X))=\dim
Gr_F^p Gr_{p+q}^W H^k(X).$$}
\end{definition}

We are ready now to define the Hodge--Poincar\'{e} polynomial.

\begin{definition}\textnormal{For any complex algebraic variety $X$,
we define its \emph{Hodge--Poincar\'{e} polynomial} as
$$HP(X)(u,v)=\sum_{p,q}(-1)^{p+q}\chi_{p,q}(X)u^p v^q =\sum_{p,q,k}(-1)^{p+q+k}h^{p,q}(H^k(X))u^p v^q .$$}
\end{definition}

\begin{remark} \label{HPtoHD}\textnormal{When our algebraic variety $X$ is smooth, Poincar\'{e} duality gives us the
following functional identity relating Hodge--Deligne and
Hodge--Poincar\'{e} polynomials
\begin{equation}\label{identityDP}\mathcal{H}(X)(u,v)=(uv)^{\dim_{\mathbb{C}}
X}\cdot HP(X)(u^{-1},v^{-1})
\end{equation}where $\dim_{\mathbb{C}}
X$ denotes the complex dimension of $X$.}

\textnormal{Let $b^k(X)=\dim H^k(X)$ be the $k$--Betti number of
the variety $X$ and let $P_X(t)=\sum_k b^k(X)t^k$ be its
Poincar\'e polynomial. If $X$ is not only smooth, but also
projective, the Betti numbers of $X$ satisfy
\begin{equation}\label{betti}b^k(X)=\sum_{p+q=k}h^{p,q}(H^k(X))\end{equation}so
that
\begin{equation}\label{identityDP2}P_X(t)=\sum_k b^k(X)t^k
=\mathcal{H}(X)(t,t)=HP(X)(t,t).\end{equation}}
\end{remark}

Hodge--Deligne polynomials are very useful because of their rather
nice properties. Now we introduce some results that will be quite
helpful to do our computations. In \cite{Du} Durfee proved that if
$X=\cup_i X_i$ and $Y=\cup_i Y_i$ are smooth projective varieties
that are disjoint unions of locally closed subvarieties, such that
$X_i \cong Y_i$ for all $i$, then $X$ and $Y$ have the same Betti
numbers. Using the properties of Hodge--Deligne polynomials, in
particular their relation with virtual Poincar\'{e} polynomials,
one may prove that this is also true for the Hodge numbers of $X$
and $Y$. Here we are using the following extension of Durfee's result

\begin{theo}[\cite{MOV1}, Theorem 2.2]\label{Theorem2.2}Let $X$ be a complex
variety. Suppose that $X$ is a finite disjoint union $X=\cup_i
X_i$, where $X_i$ are locally closed subvarieties. Then
$$\mathcal{H}(X)(u,v)=\sum_i\mathcal{H}(X_i)(u,v).$$
\end{theo}

Another result from \cite{MOV1} that will be useful for our
computations when we are dealing with fibrations is
\begin{lemma}[\cite{MOV1}, Lemma 2.3]\label{lema2.3}Suppose that $\pi :X\rightarrow Y$
is an algebraic fiber bundle with fiber $F$ which is locally
trivial in the Zariski topology, then
$$\mathcal{H}(X)(u,v)=\mathcal{H}(F)(u,v)\cdot\mathcal{H}(Y)(u,v).$$
\end{lemma}

In this paper we consider varieties acted on by algebraic groups.
Then, we need a cohomology theory that captures all the
information given by the action of the group. Namely
equivariant cohomology. Hodge--Poincar\'{e} polynomials can
be extended to analogous polynomials for equivariant cohomology
groups. We shall call this new series the equivariant
Hodge--Poincar\'{e} series.

If $X$ is an algebraic variety acted on by a group $G$, consider
$EG \rightarrow BG$ a universal classifying bundle for $G$, where
$BG=EG/G$ is the \emph{classifying space} of $G$ and $EG$ is the
\emph{total space} of $G$. We form the space $X\times_G EG$ which
is defined to be the quotient space of $X\times EG$ by the
equivalence relation $(x,e\cdot g)\sim (g\cdot x,e)$. Then, the
\emph{equivariant cohomology ring} of $X$ is the following
$$H^{\ast}_G(X)=H^{\ast}(X\times_G EG).$$ Although $EG$ and $BG$
are not finite-dimensional manifolds, there are natural Hodge
structures on their cohomology. This is trivial in the case of
$EG$. Deligne proved that there is a pure Hodge structure on
$H^{\ast}(BG)$ and that $H^{p,q}(H^{\ast}(BG))=0$ for $p\neq q$
(see \cite{D4} \S 9). We may regard $EG$ and $BG$ as increasing
unions of finite-dimensional varieties $(EG)_m$ and $(BG)_m$ for
$m\geq 1$ such that $G$ acts freely on $(EG)_m$ with
$(EG)_m/G=(BG)_m$ and the inclusions of $(EG)_m$ and $(BG)_m$ in
$EG$ and $BG$ respectively induce isomorphisms of cohomology in
degree less than $m$ which preserve the Hodge structures. In the
same way $X\times_G EG$ is the union of finite-dimensional
varieties whose natural mixed Hodge structures induce a natural
mixed Hodge structure on $H^n (X\times_G EG)$. Using that we have
the following

\begin{definition}\textnormal{We define the \emph{equivariant Hodge--Poincar\'{e} numbers}
of $X$ as $$h^{p,q;n}_G(X)=h^{p,q}(H^n (X\times_G
EG)).$$}\end{definition}

We are ready now to define the equivariant Hodge--Poincar\'{e}
series.

\begin{definition}\textnormal{For any complex algebraic variety $X$ acted on by an algebraic group $G$,
we define its \emph{equivariant Hodge--Poincar\'{e} series} as
$$HP_G(X)(u,v)=\sum_{p,q,k}(-1)^{p+q+k}h_G^{p,q;k}(X)u^p v^q .$$}
\end{definition}

\begin{parrafo}\label{parraequicoho}\textnormal{Suppose now that $G$ is connected. The relationship
between cohomology and equivariant cohomology is accounted for by
a Leray spectral sequence for the
fibration\begin{equation}\label{FEC}X\times_G EG \rightarrow BG
\end{equation}whose fiber is $X$. The $E_2$-term of this spectral sequence is given
by $E_2^{p,q}=H^p(X)\otimes H^q(BG)$ which abuts to
$H^{p+q}_G(X)$. This spectral sequence preserves Hodge structures.}

\textnormal{If $X$ is a nonsingular projective variety that is acted on linearly by a connected complex reductive group $G$, one has that the fibration
(\ref{FEC}) is cohomologically trivial over $\mathbb{Q}$ (see
\cite{K1} Proposition 5.8). Then
\begin{equation}\label{cohoeq}H^{\ast}_G(X)\cong H^{\ast}(X)\otimes
H^{\ast}(BG).\end{equation}This isomorphism is actually an
isomorphism of mixed Hodge structures (\cite{D4} Proposition 8.2.10).}

\textnormal{We have another fibration, that is $$X\times_G EG
\rightarrow X/G$$with fiber $EG$. When $G$ acts freely on $X$,
that is the stabilizer of every point is trivial, then it induces
the isomorphism\begin{equation}\label{quotient}H^{\ast}(X\times_G
EG)\cong H^{\ast}(X/G).\end{equation}Hence, if $X$ is finite-dimensional and $G$ acts freely on it, ${HP}_G(X)(u,v)$ is a polynomial.}\end{parrafo}

We need the following result from \cite{GM} for future
computations.
\begin{lemma}\label{lemmafibrationEHP}Let $Y\rightarrow Z$ be
a locally trivial fibration in the Zariski topology with fibre
$F$, and such that it is compatible with respect to the action of
the group $G$ that acts on $Y$ and $Z$ respectively. Assume that $Y$ and $Z$ are smooth varieties. Then
$$HP_G(Y)(u,v)=HP_G(Z)(u,v)\cdot HP(F)(u,v).$$\end{lemma}

We are ready now to compute the Hodge--Deligne polynomials of our
strata. In the rest of the section we will describe how we can do
it for the case in which we have two components in the type we use
to define the stratification. When
${\underline{n}}=(n_1,n-k-n_1)$, we proved that the stratum can be
described as a complement of a determinantal variety. Our strategy
could be understood by looking at what happens at the stratum
$\mathscr{W}_{\mathcal{E}_{\underline{n}}}$ when $n_1 \neq
\frac{1}{2}(n-k)$. The remaining cases are analogous.

\begin{theo}\label{HodgeStratum}Using the notations of Subsection
\ref{determinantal}, the stratum
$\mathscr{W}_{\mathcal{E}_{\underline{n}}}$ for the type
$\underline{n}=(n_1,n-k-n_1)$ has the following Hodge--Poincar\'e
polynomial
\begin{align*}{HP}(\mathscr{W}_{\mathcal{E}_{\underline{n}}})(u,v)&=
{HP}(\mathcal{M}(n_1,d_1))(u,v)\cdot
{HP}(\mathcal{M}(n-k-n_1,d-d_1))(u,v) \\& \cdot
\frac{1-(uv)^{n_1\cdot (n-k-n_1)\cdot (g-1)}}{1-uv} \cdot\bigg{[}
\frac{(1-(uv)^{N-k+1})\cdot \ldots \cdot (1-(uv)^{N})}{(1-uv)\cdot
\ldots \cdot (1-(uv)^{k})}- \\& -\sum_{\mu =\lceil k (1-\frac{n_1}{n-k})\rceil }^{\min \{ k,j
\}}(uv)^{\mu(N-k-j+{\mu})}\cdot \frac{(1-(uv)^{N-j-k+\mu+1})\cdot \ldots \cdot
(1-(uv)^{N-j})}{(1-uv)\cdot \ldots \cdot (1-(uv)^{k-\mu})} \cdot \\& \cdot \frac{(1-(uv)^{j-\mu+1})\cdot \ldots \cdot
(1-(uv)^{j})}{(1-uv)\cdot \ldots \cdot (1-(uv)^{\mu})} \bigg{]},
\end{align*}where $N=d+(n-k)(g-1)$ and
$j=d-d_1+(n-k-n_1)(g-1)$. The numbers $d_1$ and $d-d_1$ must
satisfy the following identity:
$$\frac{d_1}{n_1}=\frac{d-d_1}{n-k-n_1}.$$\end{theo}

\begin{proof}In Remark \ref{fibrationStrata} we saw that
${W}_{\mathcal{E}_{\underline{n}}}$ may be described as a locally
trivial fiber bundle (in the Zariski topology) over $\mathbb{P}(\mathcal{H}^s)$, where
$\mathbb{P}(\mathcal{H}^s)$ is the projective fibration over
$\mathcal{R}^s_1 \times \mathcal{R}^s_2$ that appears in the proof
of Proposition \ref{sequenceof projective}, with fiber the
complement of ${V}$, we denote it by ${V}^c$, in
$\Gr(k,d+(n-k)(g-1))$. We also saw that both fibrations are
$PGL(N_1)\times PGL(N_2)$-invariant, for certain $N_1$ and $N_2$. Note that
$\mathbb{P}(\mathcal{H}^s)$ is actually a projective fibration
with fiber the projective space of dimension
$n_1(n-k-n_1)(g-1)-1$. We label
$N=h^1(F^{\vee})=d+(n-k)(g-1)$. Using then Lemma
\ref{lemmafibrationEHP} we have that
\begin{align}\label{53}&{HP}_{PGL(N_1)\times PGL(N_2)}({W}_{\mathcal{E}_{\underline{n}}})(u,v)
={HP}_{PGL(N_1)\times PGL(N_2)}(\mathbb{P}(\mathcal{H}^s))(u,v)\cdot {HP}({V}^c)(u,v)=\\&
={HP}_{PGL(N_1)\times PGL(N_2)}(\mathcal{R}^s_1 \times
\mathcal{R}^s_2)(u,v){HP}(\mathbb{P}^{n_1\cdot (n-k-n_1) \cdot
(g-1)-1})(u,v) \cdot {HP}({V}^c)(u,v)\nonumber.
\end{align}Now, the varieties ${W}_{\mathcal{E}_{\underline{n}}}$ and $\mathcal{R}^s_1 \times
\mathcal{R}^s_2$ are closed under the action of $PGL(N_1)\times
PGL(N_2)$. This group is connected and the action is actually free
then the stabilizers are trivial. Then we may apply paragraph
\ref{parraequicoho}. We obtain that identities (\ref{cohoeq}) and
(\ref{quotient}) hold, then
\begin{align*}H^{\ast}_{PGL(N_1)\times PGL(N_2)}({W}_{\mathcal{E}_{\underline{n}}})& \cong
H^{\ast}({W}_{\mathcal{E}_{\underline{n}}}/PGL(N_1)\times
PGL(N_2))\cong
H^{\ast}(\mathscr{W}_{\mathcal{E}_{\underline{n}}})\end{align*}
and
\begin{align*}H^{\ast}_{PGL(N_1)\times
PGL(N_2)}(\mathcal{R}^s_1 \times \mathcal{R}^s_2)& \cong
H^{\ast}(\mathcal{R}^s_1 \times \mathcal{R}^s_2/PGL(N_1)\times
PGL(N_2))\cong \\& \cong H^{\ast}(\mathcal{M}(n_1,d_1)\times
\mathcal{M}(n-k-n_1,d-d_1))\cong \\& \cong
H^{\ast}(\mathcal{M}(n_1,d_1))\otimes
H^{\ast}(\mathcal{M}(n-k-n_1,d-d_1)),\end{align*}using K\"unneth
formula. These are isomorphisms of mixed Hodge structures, so induce
the following identity of Hodge--Poincar\'e polynomials
\begin{align*}HP_{PGL(N_1)\times
PGL(N_2)}({W}_{\mathcal{E}_{\underline{n}}})(u,v)=
HP(\mathscr{W}_{\mathcal{E}_{\underline{n}}})(u,v)\end{align*} and
\begin{align*}HP_{PGL(N_1)\times
PGL(N_2)}(\mathcal{R}^s_1 \times
\mathcal{R}^s_2)(u,v)=HP(\mathcal{M}(n_1,d_1))(u,v)\cdot
HP(\mathcal{M}(n-k-n_1,d-d_1))(u,v).\end{align*}Now, the
Hodge--Poincar\'e polynomial of the projective space is
${HP}(\mathbb{P}^n)(u,v)=\frac{1-(uv)^{n+1}}{1-uv}$ for every $n$.
Substituting these in (\ref{53}) one obtains the following identity
of Hodge--Poincar\'e polynomials
\begin{align*}{HP}&(\mathscr{W}_{\mathcal{E}_{\underline{n}}})(u,v)=\\&
= {HP}(\mathcal{M}(n_1,d_1))(u,v)\cdot
{HP}(\mathcal{M}(n-k-n_1,d-d_1))(u,v)  \cdot
\frac{1-(uv)^{n_1\cdot (n-k-n_1)\cdot (g-1)}}{1-uv} \cdot
HP(V^c)(u,v).
\end{align*}

Regarding $HP(V^c)(u,v)$, the variety $V= V_w$ where $w$ is a
point in $W_{\mathcal{E}_{\underline{n}}}:= \{(e_1, e_2, e)$ where
$(e_1 ,e_2 )\in \mathcal{R}^s_1\times \mathcal{R}^s_2$ and $e\in
\mathbb{P}(\mathcal{H}^s_{(e_1 ,e_2)})\}$ using the notations of
Subsection \ref{determinantal}. The variety $V$ is actually
independent of the point $w$ and is equal to
\begin{equation*}V=V_w:=\big{\{} \pi \in \Gr(k,\mathcal{R}^1
p_{{\mathbb{P}(\mathcal{H}^s)}{\ast}}\mathcal{F}^{s\vee})_w :
\dim\big{(}\pi \cap (\pi_P^{s\ast}p_2^{s\ast}\mathcal{R}^1
\pi^s_{2\ast}\mathcal{U}^{s\vee}_2)_w\big{)}\geqslant k
(1-\frac{n_1}{n-k}) \big{\}},
\end{equation*}let $(\pi_P^{s\ast}p_2^{s\ast}\mathcal{R}^1
\pi^s_{2\ast}\mathcal{U}^{s\vee}_2)_w =H^1 (Q_2^{\vee})$ and
$j=h^1 (Q_2^{\vee})=d-d_1+(n-k-n_1)(g-1)$. Analogously we denote
$\pi_P^{\ast}p_1^{s\ast}\mathcal{R}^1\pi^s_{1\ast}\mathcal{U}_1^{s\vee}
\otimes \mathcal{O}_P(-1)=H^1(F_1^{\vee})$ and
$h^1(F_1^{\vee})=n_1(g-1)+d_1$. Then $V$ can be written as \begin{equation}\label{junio2}
V:=\coprod_{\mu =\lceil k (1-\frac{n_1}{n-k})\rceil }^{\min \{
k,j \}}\big{\{} \pi \in \Gr(k,N) : \dim\big{(}\pi \cap H^1
(Q_2^{\vee})\big{)}=\mu \big{\}},\end{equation}and denote $V^{\mu}:= \big{\{}
\pi \in \Gr(k,N) : \dim\big{(}\pi \cap H^1 (Q_2^{\vee})\big{)}=\mu
\big{\}}$ for integers $\mu$ between $\lceil k
(1-\frac{n_1}{n-k})\rceil$ and $\min \{ k,j \}$.

For every $\mu$, the variety $V^{\mu}$ is isomorphic to a fibration over $\Gr (k-\mu,N-j)\times \Gr(\mu ,j)$ with fibre $\mathbb{C}^{(j-\mu)(k-\mu)}$. Then, we have the following identity of Hodge--Deligne polynomials\begin{equation}\label{Junio}\mathcal{H}(V^{\mu})(u,v)=\mathcal{H}(\Gr (k-\mu,N-j))(u,v)\cdot \mathcal{H}( \Gr(\mu ,j))(u,v)\cdot \mathcal{H}(\mathbb{C}^{(j-\mu)(k-\mu)})(u,v).\end{equation}
Now, from Remark \ref{HPtoHD} one has that
$HP(V^c)(u,v)=(uv)^{\dim_{\mathbb{C}}V^c}\mathcal{H}(V^c)(u^{-1},v^{-1})$.
Using now Theorem \ref{Theorem2.2} and applying again the previous
identity relating Hodge--Poincar\'e and Hodge--Deligne
polynomials, we obtain
\begin{align}\label{HPno me acuerdo}HP(V^c)&(u,v)=
(uv)^{\dim_{\mathbb{C}}V^c} \mathcal{H}(V^c)(u^{-1},v^{-1})=
 \\& =(uv)^{\dim_{\mathbb{C}}V^c}\Big{[} \mathcal{H}(\Gr
(k,N))(u^{-1},v^{-1})- \sum_{\mu =\lceil k
(1-\frac{n_1}{n-k})\rceil }^{\min \{ k,j
\}}\mathcal{H}(V^{\mu})(u^{-1},v^{-1})\Big{]}= \nonumber
\\& =(uv)^{\dim_{\mathbb{C}}V^c}\bigg{[} (uv)^{-\dim_{\mathbb{C}}\Gr
(k,N)}HP(\Gr (k,N))(u,v)- \nonumber \\& -\sum_{\mu =\lceil k
(1-\frac{n_1}{n-k})\rceil }^{\min \{ k,j \}}(uv)^{-k(N-k)+\mu(N-k-j+\mu)}\Big{[}{HP}(\Gr (k-\mu,N-j))(u,v)\cdot \nonumber \\& \cdot HP( \Gr(\mu ,j))(u,v)\cdot HP(\mathbb{C}^{(j-\mu)(k-\mu)})(u,v)\Big{]}\bigg{]}.\nonumber\end{align}The Grassmannian $\Gr
(k,N)$ is a smooth projective variety.
Note that the Hodge--Poincar\'e polynomial of the Grassmannian,
$HP(\Gr (k,N))(u,v)$, is rather simple. The cohomology of the
Grassmannian is integral, hence only types $(p,p)$ occur. This
fact implies that the identity (\ref{betti}) is in this case the
following$$b^{2p}(\Gr (k,N))=h^{p,p}(H^{2p}(\Gr (k,N)))$$so
\begin{align}\label{HodgeGrass}{HP}(\Gr (k,N))(u,v)&=\sum_p h^{p,p}(H^{2p}(\Gr
(k,N)))u^p v^p = \frac{(1-(uv)^{N-k+1})\cdot \ldots \cdot
(1-(uv)^{N})}{(1-uv)\cdot \ldots \cdot (1-(uv)^{k})}.
\end{align}Moreover, the variety $V^c$ is open in $\Gr (k,N)$ then
they have the same dimension. In addition, it is not difficult to
see that ${HP}(\mathbb{C}^m)(u,v)=1$ for all $m$. Substituting
these in (\ref{HPno me acuerdo}), we
get\begin{align*}HP&(V^c)(u,v)=\frac{(1-(uv)^{N-k+1})\cdot \ldots \cdot
(1-(uv)^{N})}{(1-uv)\cdot \ldots \cdot (1-(uv)^{k})}- \\& -\sum_{\mu =\lceil k (1-\frac{n_1}{n-k})\rceil }^{\min \{ k,j
\}}(uv)^{\mu(N-k-j+{\mu})}\cdot \frac{(1-(uv)^{N-j-k+\mu+1})\cdot \ldots \cdot
(1-(uv)^{N-j})}{(1-uv)\cdot \ldots \cdot (1-(uv)^{k-\mu})} \cdot \\& \cdot \frac{(1-(uv)^{j-\mu+1})\cdot \ldots \cdot
(1-(uv)^{j})}{(1-uv)\cdot \ldots \cdot (1-(uv)^{\mu})}.\end{align*}Then we conclude.
\end{proof}

\begin{remark}\textnormal{Regarding the Hodge--Poincar\'e polynomial of the moduli
space of stable bundles of rank $n$ and degree $d$ not everything
is known. For $\gcd(n,d)=1$, the expression for ${HP}(\mathcal{M}(2,
d))(u,v)$ can be deduced from Peter Newstead's article \cite{N1},
although it did not appear written out in this paper. The first time
that this appeared in the literature is in the article \cite{Ba} by
S. del Ba\~no Roll\'{i}n. In \cite{EK} R. Earl and F. Kirwan give
an inductive formula for the Hodge--Poincar\'e polynomials of this
moduli spaces, and in particular they compute it explicitly for
some cases with rank different to $2$. When $\gcd(n,d)\neq 1$ the
Hodge--Deligne polynomial $\mathcal{H}(\mathcal{M}(2,d))(u,v)$
where $\mathcal{M}(2,d)$ is the moduli space of stable vector
bundles of rank $2$ and even degree, has been recently computed by
Mu\~noz \emph{et al.} (see \cite{MOV2} Theorem 5.2) using its
relation with certain moduli spaces of triples and by myself in
\cite{GM}.}
\end{remark}
\subsection{Explicit computations for $n-k=2$}
Under this hypothesis we see that our coherent systems $(E,V)$ of
type $(n,d,k)$ are coming from BGN extensions whose quotient
bundle $F$ has rank $2$. Then the subbundles $Q_1$ and $Q_2$ are
actually line bundles, hence the type in this case is
$\underline{n}=(n_1,n-k-n_1)=(1,1)$. Bearing in mind the equality
of the slopes, the degrees satisfy that $d_1=d/2=d-d_1$. Using
the notations of Definition \ref{strata}, we have the following
decomposition
\begin{equation}\label{Decompostion}G_L (n,d,k)=\mathscr{W}^1 \sqcup
\mathscr{W}_{\mathcal{E}_{\underline{n}}} \sqcup
\mathscr{W}_{\mathcal{E}_{\underline{n}}'} \sqcup
\mathscr{W}_{\mathcal{SE}_{\underline{n}}}\sqcup
\mathscr{W}_{\mathcal{SE}_{\underline{n}}'}.\end{equation}Here
$\mathscr{W}^1$ is the open stratum and classifies the coherent
systems coming from a BGN extension of quotient being stable. The
stratum $ \mathscr{W}_{\mathcal{E}_{\underline{n}}}$ classifies
the cases in which the quotient bundle is the bundle in the middle
of an extension of the following type
\begin{equation}\label{line}0\rightarrow L \rightarrow F \rightarrow
L'\rightarrow 0\end{equation}that is a nonsplit extension and the
line bundles $L$ and $L'$ are nonisomorphic. In the same fashion
$\mathscr{W}_{\mathcal{E}_{\underline{n}}'}$ classifies the cases
in which (\ref{line}) satisfies that $L \cong L'$. The varieties
$\mathscr{W}_{\mathcal{SE}_{\underline{n}}}$ and
$\mathscr{W}_{\mathcal{SE}_{\underline{n}}'}$ are as before but
for the bundle $F$ being split.

We have two
different cases when $(n-k,d)=(2,d)$, either $\gcd(2,d)=1$ or $\gcd(2,d)\neq 1$.
The computations for these cases are done in the following
theorems.

\begin{theo}\label{H1}The Hodge--Deligne polynomial of the moduli space
$G_L(n,d,k)$ for $n-k=2$ and $d$ odd is
\begin{align*}\mathcal{H}(G_L(n,d,k))(u,v)
=&(1+u)^g(1+v)^g\cdot
\frac{(1+u^2v)^g(1+uv^2)^g-u^gv^g(1+u)^g(1+v)^g}{(1-uv)(1-u^2v^2)}\cdot
\\& \cdot \frac{(1-(uv)^{2(g-1)+d-k+1})\cdot \ldots \cdot
(1-(uv)^{2(g-1)+d})}{(1-uv)\cdot \ldots \cdot
(1-(uv)^{k})}\end{align*}\end{theo}
\begin{proof}Using Proposition \ref{prop:restate} we have that when
$\gcd(n-k,d)=1$, then $\mathscr{W}_{\mathcal{E}_{\underline{n}}}$,
$\mathscr{W}_{\mathcal{E}_{\underline{n}}'}$,
$\mathscr{W}_{\mathcal{SE}_{\underline{n}}}$, and
$\mathscr{W}_{\mathcal{SE}_{\underline{n}}'}$ are all empty and
$G_L(n,d,k)$ is actually a Grassmann fibration on
$\mathcal{M}(n-k,d)$ with fiber the Grassmannian $\Gr
(k,d+(n-k)(g-1))$. Using now Lemma \ref{lema2.3} we get
that$$\mathcal{H}(G_L(n,d,k))(u,v) =\mathcal{H}(\Gr (k, d
+2(g-1)))(u,v)\cdot \mathcal{H}(\mathcal{M}(2,d))(u,v).$$We
already know what $\mathcal{H}(\Gr (k, d +2(g-1)))(u,v)$ looks
like, the computation appears in (\ref{HodgeGrass}). Regarding
$\mathcal{H}(\mathcal{M}(2,d))(u,v)$, for $d$ odd, using Lemma 3
and Corollary 5 of \cite{EK} we get that
$$\mathcal{H}(\mathcal{M}(2,d))(u,v)=(1+u)^g(1+v)^g\cdot
\frac{(1+u^2v)^g(1+uv^2)^g-u^gv^g(1+u)^g(1+v)^g}{(1-uv)(1-u^2v^2)},$$so
we conclude.
\end{proof}
\begin{theo}\label{H2}The Hodge--Deligne polynomial of the moduli space
$G_L(n,d,k)$ for $(n,d,k)=(3,d,1)$, $d$ even and $g\geq \frac{3-d}{2}$ is
\begin{align*}
\mathcal{H}(G_L (3,d,1&))(u,v)=
\frac{(1 + u)^g (1 + v)^g(uv-(uv)^{\frac{d}{2}+g})}{u^3v^3(uv-1)^3(uv+1)}\Big{(}
 (uv)^{3+g} (1 + u)^g(1 + v)^g + 
\\& +(uv)^{\frac{d}{2}+2g}(1 + u)^g(1 + v)^g
-(uv)^{\frac{d}{2}+g+1}(1 + u^2v)^g(1 + uv^2)^g
-(uv)^{2}(1 + u^2v)^g(1 + uv^2)^g
\Big{)}.\end{align*}
\end{theo}

\begin{proof}Applying Theorem \ref{Theorem2.2} to (\ref{Decompostion}) we
obtain the following identity \begin{align*}\mathcal{H}(G_L
&(n,d,k))(u,v)=\\& =\mathcal{H}(\mathscr{W}^1)(u,v)+
\mathcal{H}(\mathscr{W}_{\mathcal{E}_{\underline{n}}})(u,v)+
\mathcal{H}(\mathscr{W}_{\mathcal{E}_{\underline{n}}'})(u,v)+
\mathcal{H}(\mathscr{W}_{\mathcal{SE}_{\underline{n}}})(u,v)+
\mathcal{H}(\mathscr{W}_{\mathcal{SE}_{\underline{n}}'})(u,v).\end{align*}As
we did in the proof of Theorem \ref{H1}, Proposition
\ref{prop:restate} tells us that when $\gcd(n,d,k)=1$, $\mathscr{W}^1$
is a Grassmann fibration on $\mathcal{M}(n-k,d)$ with fiber the
Grassmannian $\Gr (k,d+(n-k)(g-1))$. Here, $(n-k,d)=(2,d)$ and $\gcd(2,d)\neq 1$ then the Grassmann fibration is constructed at the Quot-scheme level since there is no Poincar\'e bundle over $\mathcal{M}(2,d)$ (see \cite{BG}, Proposition 4.4). This fibration at the Quot-scheme level is locally trivial in the Zariski topology, then, if we denote $R^s$ the corresponding set of stable points and $W$ the set corresponding to $\mathscr{W}^1$, at the Quot-scheme level, from Lemma \ref{lemmafibrationEHP} we have that $${HP}_{PGL(N)}({W})(u,v)
={HP}(\Gr (k, d +2(g-1)))(u,v)\cdot
{HP}_{PGL(N)}(R^s)(u,v).$$The action of $PGL(N)$ on $W$ and $R^s$ is free, so $${HP}(\mathscr{W}^1)(u,v)
={HP}(\Gr (k, d +2(g-1)))(u,v)\cdot
{HP}(\mathcal{M}(2,d))(u,v).$$ Moreover, by Theorem \ref{Theo} (c) we have that $\mathscr{W}^1$ is smooth, applying Remark \ref{HPtoHD} we get that$$\mathcal{H}(\mathscr{W}^1)(u,v)
=\mathcal{H}(\Gr (k, d +2(g-1)))(u,v)\cdot
\mathcal{H}(\mathcal{M}(2,d))(u,v),$$where
$\mathcal{H}(\mathcal{M}(2,d))(u,v)$ is the Hodge--Deligne
polynomial of the moduli space $\mathcal{M}(2,d)$ of stable vector
bundles of rank $2$ and even degree. This can be found in
\cite{MOV2}, Theorem 5.2. The polynomial is
\begin{align*}\mathcal{H}(\mathcal{M}(2,d))(u,v)
=& \frac{1}{2(1-uv)(1-u^2v^2)}[
2(1+u)^g(1+v)^g(1+u^2v)^g(1+uv^2)^g- \\& -
(1+u)^{2g}(1+v)^{2g}(1+2u^{g+1}v^{g+1}-u^2v^2)-(1-u^2)^g(1-v^2)^g(1-uv)^2]
\end{align*}then we obtain
\begin{align}\label{100}\mathcal{H}(\mathscr{W}^1)(u,v)
=&\frac{1}{2(1-uv)(1-u^2v^2)}[
2(1+u)^g(1+v)^g(1+u^2v)^g(1+uv^2)^g- \nonumber \\& -
(1+u)^{2g}(1+v)^{2g}(1+2u^{g+1}v^{g+1}-u^2v^2)-(1-u^2)^g(1-v^2)^g(1-uv)^2]
\cdot
\\& \cdot \frac{(1-(uv)^{2(g-1)+d-k+1})\cdot \ldots \cdot
(1-(uv)^{2(g-1)+d})}{(1-uv)\cdot \ldots \cdot (1-(uv)^{k})}.\nonumber
\end{align}In order to compute $\mathcal{H}(\mathscr{W}_{\mathcal{E}_{\underline{n}}})(u,v)$,
$\mathcal{H}(\mathscr{W}_{\mathcal{E}_{\underline{n}}'})(u,v)$,
$\mathcal{H}(\mathscr{W}_{\mathcal{SE}_{\underline{n}}})(u,v)$ and
$\mathcal{H}(\mathscr{W}_{\mathcal{SE}_{\underline{n}}'})(u,v)$ we
use Theorem \ref{HodgeStratum}. Although in these cases we
do not need to take into account the action of a group, because we
can do the construction as complements of determinantal varieties
at the moduli space level, we would need to consider the action of a group of automorphisms as described in Remark \ref{nota} and Theorems \ref{deter1} to \ref{deter4}.

Note that we can describe our strata as locally trivial fiber
bundles (Remark \ref{fibrationStrata}). The fiber is
the complement in a Grassmannian of a union of certain varieties as
one can see in the proof of Theorem \ref{HodgeStratum}. The base
space in the different locally trivial fiber bundles is the space
classifying the different types of extensions that can
appear in the case we are dealing with, see Proposition
\ref{eandep}. We use here the notation of Subsection
\ref{determinantal}.
Let ${\mathcal{E}_{\underline{n}}}$ be the space that parametrizes the
extensions
\begin{equation*}0\rightarrow L \rightarrow F \rightarrow
L'\rightarrow 0\end{equation*}that are nonsplit and such that the line
bundles $L$ and $L'$ are nonisomorphic. Let ${\mathcal{E}_{\underline{n}}'}$ be the space that
parametrizes the extensions as above where $L \cong L'$.

Now, $\mathscr{W}_{\mathcal{E}_{\underline{n}}}$ and
$\mathscr{W}_{\mathcal{E}_{\underline{n}}'}$ can be described as locally trivial fiber
bundles over $\mathcal{E}_{\underline{n}}$ and
$\mathcal{E}_{\underline{n}}'$ respectively (Remark \ref{fibrationStrata}). The fiber of $\mathscr{W}_{\mathcal{E}_{\underline{n}}}$ is explicitly computed in Theorem \ref{HodgeStratum}. From this theorem one gets
\begin{align}\label{201} \mathcal{H}(\mathscr{W}_{\mathcal{E}_{\underline{n}}})&(u,v)
=\mathcal{H}({\mathcal{E}_{\underline{n}}})(u,v)\cdot \bigg{[}
\frac{(1-(uv)^{2(g-1)+d-k+1})\cdot \ldots \cdot
(1-(uv)^{2(g-1)+d})}{(1-uv)\cdot \ldots \cdot (1-(uv)^{k})}- \nonumber \\&
-\sum_{\mu =\lceil \frac{k}{2} \rceil }^{\min \{
k,(g-1)+\frac{d}{2} \}}(u\cdot v)^{(k-{\mu})(d/2+(g-1)-\mu)}\cdot \frac{(1-(uv)^{d/2+(g-1)-k+\mu+1})\cdot \ldots \cdot
(1-(uv)^{d/2+g-1})}{(1-uv)\cdot \ldots \cdot (1-(uv)^{k-\mu})}\cdot \\& \cdot \frac{(1-(uv)^{(g-1)+d/2-\mu +1})\cdot \ldots \cdot
(1-(uv)^{(g-1)+d/2})}{(1-uv)\cdot \ldots \cdot (1-(uv)^{\mu})}\bigg{]}\nonumber
\end{align}We saw in Proposition
\ref{eandep} (ii) that $\mathcal{E}_{\underline{n}}$ is a projective bundle over
$\Jac^{d/2}X\times \Jac^{d/2}X \setminus \Delta $, where $\Delta$
is the diagonal in $\Jac^{d/2}X\times \Jac^{d/2}X$, with fiber the
projective space of dimension $g-2$. The Hodge--Deligne
polynomials of the Jacobian and the projective space are:
\begin{equation}\label{JacoProjHD}\textnormal{$\mathcal{H}(\Jac^{\delta}X)(u,v)=(1+u)^g(1+v)^g$
 $ $ $ $ and $ $ $ $ $\mathcal{H}(\mathbb{P}^n)(u,v)=\frac{1-(uv)^{n+1}}{1-uv}$,}\end{equation}for every degree
 $\delta$. Then, using Lemma \ref{lema2.3} we get
 $$\mathcal{H}({\mathcal{E}_{\underline{n}}})(u,v)=((1+u)^{2g}(1+v)^{2g}-(1+u)^{g}(1+v)^{g})\cdot \frac{1-(uv)^{g-1}}{1-uv}.$$From the previous identity and (\ref{201}), for $k=1$ we obtain
 \begin{align}\label{200} \mathcal{H}(\mathscr{W}_{\mathcal{E}_{\underline{n}}})&(u,v)
=((1+u)^{2g}(1+v)^{2g}-(1+u)^{g}(1+v)^{g})\cdot \frac{1-(uv)^{g-1}}{1-uv}\cdot
\frac{(uv)^{(g-1)+\frac{d}{2}}-(uv)^{2(g-1)+d}}{1-uv}.
\end{align}

With regards to the stratum
$\mathscr{W}_{\mathcal{E}_{\underline{n}}'}$, it can described as a locally trivial fiber
bundle over
$\mathcal{E}_{\underline{n}}'$. For a non-splitting extension$$0\rightarrow L\xymatrix{\ar^{\phi}[r]&}F\xymatrix{\ar^{\varphi}[r]&}L\rightarrow 0,$$taking dual and cohomology, one gets
$$0\rightarrow H^1(L^{\vee})\xymatrix{\ar^{\varphi^{\vee}}[r]&}H^1(F^{\vee})\xymatrix{\ar^{\phi^{\vee}}[r]&}H^1(L^{\vee})\rightarrow 0,$$then $H^1(F^{\vee})$ is non-canonically isomorphic to $H^1(L^{\vee})\oplus H^1(L^{\vee})$. For $k=1$, using Theorem \ref{bgn} and Theorem \ref{deter3} a BGN extension class $e \in  H^1(F^{\vee})$ gives rise to a $\alpha$-stable coherent system if $\phi^{\vee}(e)\neq 0$ and therefore the classes not contradicting $\alpha$-stability are those in $H^1(L^{\vee})\times (H^1(L^{\vee})-\{ 0\})$. The automorphism group of $F$, $\Aut (F) \cong \mathbb{C}\times \mathbb{C}^{\ast}$, acts on $H^1(F^{\vee})$ by $(\lambda,\mu)\cdot (e,e')=(\mu e + \lambda e', \mu e')$. From the induced action in $H^1(L^{\vee})\times (H^1(L^{\vee})-\{ 0\})$, we have that the fiber of $\mathscr{W}_{\mathcal{E}_{\underline{n}}'}$ over $\mathcal{E}_{\underline{n}}'$ is actually a locally trivial fibration over $\mathbb{P}(h^1(L^{\vee}))=\mathbb{P}^{\frac{d}{2}+g-2}$ whose fibre is $\mathbb{C}^{\frac{d}{2}+g-2}$. Then $\mathcal{H}(\mathscr{W}_{\mathcal{E}_{\underline{n}}'})(u,v)$ is given by\begin{align}\label{2001} \mathcal{H}(\mathscr{W}_{\mathcal{E}_{\underline{n}}'})&(u,v)
=\mathcal{H}({\mathcal{E}_{\underline{n}}'})(u,v)\cdot
\frac{(1-(uv)^{\frac{d}{2}+g-1})}{(1-uv) }\cdot (uv)^{\frac{d}{2}+g-2}.
\end{align}
By Proposition \ref{eandep} (ii) $\mathcal{E}_{\underline{n}}'$ is a projective bundle over
$\Jac^{d/2}X$ with fiber the projective space of dimension $g-1$,
so its Hodge--Deligne polynomial is
$$\mathcal{H}({\mathcal{E}_{\underline{n}}'})(u,v)=(1+u)^{g}(1+v)^{g}\cdot
\frac{1-(uv)^{g}}{1-uv}.$$Combining the previous identity and (\ref{2001}) we get\begin{align}\label{2002} \mathcal{H}(\mathscr{W}_{\mathcal{E}_{\underline{n}}'})&(u,v)
=(1+u)^{g}(1+v)^{g}\cdot
\frac{1-(uv)^{g}}{1-uv}\cdot
\frac{(uv)^{\frac{d}{2}+g-2}
(1-(uv)^{\frac{d}{2}+g-1})}{(1-uv) }.
\end{align}

For the splitting cases,
${\mathcal{SE}_t}$ parametrizes the split extensions
\begin{equation*}0\rightarrow L \rightarrow L\oplus L' \rightarrow
L'\rightarrow 0\end{equation*}where $L$ and $L'$ are nonisomorphic
line bundles of the same degree $d/2$. By Paragraph \ref{parrafo2}
the bundles in the middle of these extensions are classified by $(\Jac^{d/2}X \times \Jac^{d/2}X \setminus
\Delta) / (\mathbb{Z}/2)$ where $\mathbb{Z}/2$ acts on
$\Jac^{d/2}X \times \Jac^{d/2}X \setminus \Delta$ by permuting the
two factors. The stratum $\mathscr{W}_{\mathcal{SE}_{\underline{n}}}$ can be identified with a locally trivial fibration over ${\mathcal{SE}_{\underline{n}}}$ where the fiber can be described as follows. For a fixed point $w=(L,L')\in \Jac^{d/2}X \times \Jac^{d/2}X \setminus \Delta$, in Theorem \ref{deter1} we defined the
determinantal varieties
\begin{equation*}V^1_w:=\big{\{} \pi \in \Gr (k, 2(g-1)+d) :
\dim\big{(}\pi \cap H^1(L'^{\vee})\big{)}\geqslant
\frac{k}{2} \big{\}},
\end{equation*}and\begin{equation*}V^2_w:=\big{\{} \pi \in \Gr (k, 2(g-1)+d) :
\dim\big{(}\pi \cap H^1(L^{\vee})\big{)}\geqslant
\frac{k}{2} \big{\}},
\end{equation*}and saw that the fiber of the bundle $\mathscr{W}_{\mathcal{SE}_{\underline{n}}}$ is the complement in $ \Gr (k, 2(g-1)+d) $ of $V^1_w \cup V^2_w$. Note that this does not depend on $w$. From (\ref{junio2}) we have that $V^1_w$ can be written as \begin{equation*}
V^1_w:=\coprod_{\mu =\lceil k/2 \rceil }^{\min \{
k,g-1+d/2\}}\big{\{} \pi \in \Gr(k,2(g-1)+d) : \dim\big{(}\pi \cap H^1
(L'^{\vee})\big{)}=\mu \big{\}},\end{equation*}and the same is true for $V^2_w$. We write $V_1^{\mu}:= \big{\{}
\pi \in \Gr(k,2(g-1)+d) : \dim\big{(}\pi \cap H^1 (L'^{\vee})\big{)}=\mu
\big{\}}$ for integers $\mu$ between $\lceil k/2\rceil$ and $\min \{ k,g-1+d/2 \}$, and define $V_2^{\mu}$ analogously. From this analysis, when $k=1$ one has that for a fixed element $w=(L,L')\in \Jac^{d/2}X \times \Jac^{d/2}X \setminus \Delta$ the induced $\alpha$-stable coherent systems correspond to BGN extension classes given by non-zero elements $e\in H^1(L^{\vee})$ and $e' \in H^1(L'^{\vee})$, i.e. $(H^1(L^{\vee})-\{ 0 \})\times (H^1(L'^{\vee})-\{ 0 \})$. The group of automorphism of $L\oplus L'$, $\Aut (L\oplus L')=\mathbb{C}^{*}\times \mathbb{C}^{*}$, acts on $(H^1(L^{\vee})-\{0\})\times (H^1(L'^{\vee})-\{0\})$. Then, for a fixed element $w=(L,L')\in \Jac^{d/2}X \times \Jac^{d/2}X \setminus \Delta$ the fiber is $\mathbb{P}(H^1(L^{\vee}))\times \mathbb{P}(H^1(L'^{\vee}))=\mathbb{P}^{\frac{d}{2}+g-2}\times \mathbb{P}^{\frac{d}{2}+g-2}$.

Now, the action of $\mathbb{Z}/2$ on $\Jac^{d/2}X \times \Jac^{d/2}X \setminus \Delta$ by permuting the factors is reflected in the fiber $\mathbb{P}^{\frac{d}{2}+g-2}\times \mathbb{P}^{\frac{d}{2}+g-2}$ in a compatible manner. This action induces an action in cohomology in such a way that the cohomology of $\mathscr{W}_{\mathcal{SE}_{\underline{n}}}$, $H^*(\mathscr{W}_{\mathcal{SE}_{\underline{n}}})$, is identified with the invariant part of $H^*( \Jac^{d/2}X \times \Jac^{d/2}X \setminus \Delta)\otimes H^*(\mathbb{P}^{\frac{d}{2}+g-2}\times \mathbb{P}^{\frac{d}{2}+g-2})$ under the action of $\mathbb{Z}/2$ by permuting the factors.

Then, the Hodge--Deligne polynomial of $\mathscr{W}_{\mathcal{SE}_{\underline{n}}}$ is given by
\begin{align}\label{Polinomio suma}
\mathcal{H}^+(\Jac^{d/2}X\times
\Jac^{d/2}X\setminus \Delta)&(u,v)\mathcal{H}^+(\mathbb{P}^{\frac{d}{2}+g-2}\times \mathbb{P}^{\frac{d}{2}+g-2})(u,v)+ \nonumber \\& +
\mathcal{H}^-(\Jac^{d/2}X\times
\Jac^{d/2}X\setminus \Delta)(u,v)\mathcal{H}^-(\mathbb{P}^{\frac{d}{2}+g-2}\times \mathbb{P}^{\frac{d}{2}+g-2})(u,v)
\end{align}
where the subscripts $+$ and $-$ refer to the corresponding eigenspaces of eigenvalues $+1$ and $-1$ for the action of $\mathbb{Z}/2$ in both $H^{\ast}(\mathbb{P}^{\frac{d}{2}+g-2}\times \mathbb{P}^{\frac{d}{2}+g-2})$ and $H^{\ast}(\Jac^{d/2}X\times
\Jac^{d/2}X\setminus \Delta)$ respectively.

We have that $H^{\ast}(\mathbb{P}^{\frac{d}{2}+g-2}\times \mathbb{P}^{\frac{d}{2}+g-2})\cong H^{\ast}(\mathbb{P}^{\frac{d}{2}+g-2})\otimes H^{\ast}(\mathbb{P}^{\frac{d}{2}+g-2})$ and $H^{\ast}(\Jac^{d/2}X\times \Jac^{d/2}X )\cong H^{\ast}(\Jac^{d/2}X)\otimes H^{\ast}(\Jac^{d/2}X)$. These are isomorphisms of mixed Hodge structures, in fact isomorphisms of pure Hodge structures. Applying Lemma 2.6 of \cite{MOV2} to those polynomials in (\ref{Polinomio suma}), one obtains that the Hodge--Deligne polynomial of $\mathscr{W}_{\mathcal{SE}_{\underline{n}}}$ is given by
\begin{align}\label{polinomio cuarto estrato}\mathcal{H}&(\mathscr{W}_{\mathcal{SE}_{\underline{n}}})(u,v)= \nonumber \\& =
\big{[} \frac{(1-(uv)^{\frac{d}{2}+g-1})^2}{2(1-uv)^2}+\frac{(1-(uv)^{d+2(g-1)})}{2(1-(uv)^2)}\big{]}\cdot
\big{[} \frac{1}{2}(1+u)^{2g}(1+v)^{2g}+\frac{1}{2}(1-u^2)^g(1-v^2)^g-(1+u)^g(1+v)^g\big{]}+  \nonumber \\& +
\big{[} \frac{(1-(uv)^{\frac{d}{2}+g-1})^2}{2(1-uv)^2}-\frac{(1-(uv)^{d+2(g-1)})}{2(1-(uv)^2)}\big{]}\cdot
\big{[} \frac{1}{2}(1+u)^{2g}(1+v)^{2g}-\frac{1}{2}(1-u^2)^g(1-v^2)^g
\big{]} = \\& =
\big{[} \frac{(1-(uv)^{\frac{d}{2}+g-1})^2}{2(1-uv)^2}\big{]}\cdot
\big{[} (1+u)^{2g}(1+v)^{2g}-(1+u)^g(1+v)^g\big{]}+\nonumber \\& +
\big{[} \frac{(1-(uv)^{d+2(g-1)})}{2(1-(uv)^2)}\big{]}\cdot
\big{[} (1-u^2)^g(1-v^2)^g -(1+u)^g(1+v)^g
\big{]}.\nonumber
\end{align}

Finally, we consider the split
extensions in which the bundle in the middle is the direct sum of
two copies of the same line bundle of degree $d/2$\begin{equation*}0\rightarrow L \rightarrow L\oplus L \rightarrow
L\rightarrow 0.\end{equation*}The space ${\mathcal{SE}_{\underline{n}}'}$ that parametrizes the bundles $F=L\oplus L$ is identified to $\Jac^{d/2}X$. By Theorem \ref{bgn} the set of BGN extension classes giving rise to $\alpha$-stable coherent systems can be identified to the subet of $H^1(L^{\vee})\oplus H^1(L^{\vee})\cong H^1(L^{\vee})\otimes \mathbb{C}^2$ of linearly independent $e$, $e'$ in $H^1(L^{\vee})\otimes \mathbb{C}^2$. Now, the set of automorphisms of $F$, $GL(2)$, acts on $H^1(L^{\vee})\otimes \mathbb{C}^2$ via the standard representation of $GL(2)$ on $\mathbb{C}^2$. Then, the stratum $\mathscr{W}_{\mathcal{SE}_{\underline{n}}'}$ can be described as a locally trivial fibration over ${\mathcal{SE}_{\underline{n}}'}$ whose fiber at a point $L\in {\mathcal{SE}_{\underline{n}}'}\cong \Jac^{d/2}X$ is $\Gr (2, h^1(L^{\vee}))\cong \Gr (2,\frac{d}{2}+g-1)$. By Lemma \ref{lema2.3} one has that
 \begin{align}\label{Mecanso}\mathcal{H}&(\mathscr{W}_{\mathcal{SE}_{\underline{n}}'})(u,v)=\mathcal{H}({\mathcal{SE}_{\underline{n}}'})(u,v)\cdot \mathcal{H}(\Gr (2,d/2+g-1))(u,v)=\\& =\mathcal{H}(\Jac^{d/2}X)(u,v)\cdot \mathcal{H}(\Gr (2,d/2+g-1))(u,v)\nonumber =(1+u)^g(1+v)^g\bigg{[}\frac{(1-(uv)^{\frac{d}{2}+g-2})
(1-(uv)^{\frac{d}{2}+g-1})}{(1-uv)(1-(uv)^{2})}\bigg{]}.\nonumber\end{align}


Summing up polynomials (\ref{100}), (\ref{200}), (\ref{2002}), (\ref{polinomio cuarto estrato}), and (\ref{Mecanso}) together we obtain the result.\end{proof}

\begin{remark}\label{Remark_Final}\textnormal{Note that for given $(n,d,k)$ satisfying $(n-k,d)=(2,d)$, $\gcd(2,d)\neq 1$ and $k$ odd, one immediately obtains that $\gcd(n,d,k)=1$. Under this condition, the moduli space of $\alpha$-stable coherent systems, $G_L(n,d,k)$, is projective, smooth and irreducible (see \cite{KN} and Proposition \ref{prop:restate}). Then, from Remark \ref{HPtoHD} one can obtain the usual Poincar\'e polynomial of $G_L(n,d,k)$, $P_{G_L(n,d,k)}(t)$, just by knowing its Hodge--Deligne polynomial, that is $P_{G_L(n,d,k)}(t)=\mathcal{H}({G_L(n,d,k)})(t,t)$. Hence, Theorem \ref{H2} allows us to compute the Poincar\'e polynomial of $G_L(3,d,1)$ when $d$ is even and $g\geq (3-d)/2$. }\end{remark}

\begin{cor}\label{ya}
The Poincar\'e polynomial of $G_L(3,d,1)$ when $d$ is even and $g\geq (3-d)/2$ is given by
\begin{align*}P&_{G_L(3,d,1)}(t)  =\\& =\frac{(1+t)^{2 g}(-t^2 +t^{d+2g})}{{t^6} {{(-1+{t^2})}^3} (1+{t^2})}\Big({-t}^{6+2g}(1+t)^{2g} +
t^4(1+t^3)^{2g}-{t}^{4g+d}(1+t)^{2g}+t^{2+d+2g}(1+t^3)^{2g}\Big).
\end{align*}
\end{cor}

\begin{example}\textnormal{For $(n,d,k)$ satisfying the conditions of Remark \ref{Remark_Final}, one can deduce that the Poincar\'e polynomial of $G_L(n,d,k)$ should be symmetric reflecting Poincar\'e duality.}

\textnormal{For $(n,d,k)=(3,2,1)$ and $g=2$ one can easily check that $G_L(3,2,1)$ satisfies the conditions of Remark \ref{Remark_Final}. From Corollary \ref{ya} and using \textbf{Mathematica} to simplify the calculation, one obtains the following Poincar\'e polynomial:\begin{align*}P_{G_L(3,2,1)}(t)=&\mathcal{H}(G_L(3,2,1))(t,t)=1 + 4 t + 8 t^2 + 16 t^3 + 33 t^4 + 56 t^5 + 84 t^6 + 116 t^7 +\\& +
132 t^8 + 116 t^9 + 84 t^{10} + 56 t^{11} + 33 t^{12} + 16 t^{13} + 8 t^{14} + 4 t^{15} + t^{16}.\end{align*}Note that this polynomial is actually symmetric as expected.}
\end{example}

\begin{remark}\textnormal{The space $G_L(3,d,1)$ is isomorphic to the moduli space of rank 3 pairs, $\mathcal{N}_{\sigma^-_M}(3,1,d,0)$ (see \cite{Mn} for more details). The polynomial here obtained for coherent systems coincides with that of \cite{Mn} for rank 3 pairs. However, there is a typo in the formula of Theorem 6.5 in \cite{Mn}, $n_0$ should be defined as $\Big{\lceil} \frac{\sigma+d_1+d_2}{2} \Big{\rceil}$. The two formulas differ by a factor $(1+u)^g(1+v)^g$ corresponding to the fact that in \cite{Mn}, the determinant of the bundles is not fixed.}\end{remark}

\subsection*{Acknowledgements}I would like to thank Prof. Peter E. Newstead for bringing my attention to coherent systems and for his support, encouragement and advice during the research performed which ended up in this article that has been extremely enriched by his guidance. I also would like to thank Prof. Montserrat Teixidor i Bigas for her help and advice during the completion of my Ph.D. dissertation at Tufts, of which this paper is part. Thank you to the departments of mathematics of the Universities of Liverpool, Oxford and Tufts.

\end{document}